\title{On Kummer-like surfaces  attached to singularity and  modular forms}
\author{Atsuhira Nagano and Hironori Shiga}

\documentclass[10pt]{article}
\usepackage{mathrsfs,bm,graphics,amsfonts,amsthm,amssymb,amsmath,amscd,booktabs}
\usepackage{mathrsfs,bm,graphics,amsfonts,amsthm,amssymb,amsmath,amscd,booktabs}
\usepackage[dvips]{graphicx,color}
\usepackage[all]{xy} 
\def\bigzerou{\smash{\lower1.7ex\hbox{\b 0}}}

\newtheorem{thm}{Theorem}[section]

\newtheorem{lem}{Lemma}[section]
\newtheorem{prop}{Proposition}[section]
\newtheorem{rem}{Remark}[section]

\newtheorem{cor}{Corollary}[section]

\def\comment#1{{ }}
\makeatletter
  
      \@addtoreset{equation}{section}
 \usepackage{fancybox}
\begin{document}
\maketitle
\setlength{\baselineskip}{13 pt}
 \renewcommand{\thefootnote}{\fnsymbol{footnote}}

\begin{abstract}
We study a family of lattice polarized $K3$ surfaces
which is an extension of the family of Kummer surfaces derived from principally polarized Abelian surfaces.
Our family has two special properties.
First, it is coming from a resolution of a simple $K3$ singularity.
Second, it has a natural parametrization by Hermitian modular forms of four complex variables.
In this paper, we show two results:
(1) We determine the transcendental lattice and the N\'eron-Severi lattice of a generic member of our family.
(2) We give a detailed description of the double covering structure associated with our $K3$ surfaces.
\end{abstract}

\footnote[0]{Keywords:  Kummer surfaces ; $K3$ singularities ; Hermitian modular forms.  }
\footnote[0]{Mathematics Subject Classification 2020:  Primary 14J28 ; Secondary    32S25, 11F11.}
%\footnote[0]{Running head: On Kummer-like surfaces}
\setlength{\baselineskip}{14 pt}

\section*{Introduction}
A Kummer surface ${\rm Kum}(\mathfrak{A})$ for  a principally polarized Abelian surface $\mathfrak{A}$  is an algebraic $K3$ surface with an explicit defining equation parametrized by Siegel modular forms of degree $2$ (see \cite{CD} Section 5.1). 
Let $\mathcal{G}_{\rm Kum}$ be the family of such Kummer surfaces.
The Picard number of a generic member of 
$\mathcal{G}_{\rm Kum}$ is $17$.
Since $\mathcal{G}_{\rm Kum}$ 
is important in various areas in  mathematics (for example, see \cite{BHPV} or \cite{PS}),
it is meaningful to extend $\mathcal{G}_{\rm Kum}$.
In this paper,
we study a family  $\mathcal{G}_1$,
which is a natural extension of $\mathcal{G}_{\rm Kum}$  as shown in the diagram (\ref{Diagram}).

Now, 
we give a brief summary of our family $\mathcal{G}_1$.
The family $\mathcal{G}_1$ has a particular property:
it is coming from a resolution of a simple $K3$ singularity (see Section 1).
It has the following explicit expression.
Letting $T$ be a Zariski open set of the weighted projective space $\mathbb{P}(4,6,10,12,18)$,
this family is given by
$
\pi_{\mathcal{G}_1} : \mathcal{G}_1=\{K(t) |\hspace{1mm} [t]=(t_4:t_6:t_{10}:t_{12}:t_{18})\in T\} \rightarrow T,
$
where $K(t)$ is an elliptic surface given by the Weierstrass equation
\begin{align}\label{K}
 K(t):   Z^2=Y^3+a_K(U)Y +b_K(U)  \text{ with } a_K(U)= \Big(t_4 + \frac{t_{10}}{U^2} \Big),  b_K(U)= \Big(  t_6 +U^2+ \frac{ t_{12}}{U^2}+\frac{t_{18}}{U^4} \Big),
 \end{align}
which is birationally equivalent to a $K3$ surface of Picard number $16$
 (for detail, see (\ref{E88}) and Section 3.2).
Indeed,
the subfamily $\pi_{\mathcal{G}_1}^{-1} \left(\{t_{18}=0\} \cap T \right)$ coincides with the family $\mathcal{G}_{{\rm Kum}}$.
So,
we regard $\mathcal{G}_1$ as an extension of $\mathcal{G}_{\rm Kum}$
and we call a member of $\mathcal{G}_1$ a Kummer-like surface.

In order to  study  $\mathcal{G}_1$, 
we  use another family
$
\pi_{\mathcal{F}_1} : \mathcal{F}_1=\{S(t)|\hspace{1mm} [t]\in T\} \rightarrow T
$
of elliptic surfaces,
where $S(t)$ is given by the Weierstrass equation 
\begin{align}\label{S}
S(t):  Z^2=Y^3+a_S(X)Y+b_S(X) \text{ with }a_S(X)= \Big(t_4 + \frac{t_{10}}{X} \Big), b_S(X)=\Big(  t_6 +X+ \frac{ t_{12}}{X}  +\frac{t_{18}}{X^2}\Big).
\end{align}
Here, $S(t)$ gives a $K3$ surface (for detail, see (\ref{S(t)}) and Section 3.2).
For a fixed point $[t] \in T$,
 $K(t)$ is a rational double covering of $S(t)$.
We call $S(t)$ the partner surface of $K(t)$
(for this terminology, see Remark \ref{RemPartner}).
The  surface $S(t)$ has a simple structure as a lattice polarized $K3$ surface
and the period mapping for $\mathcal{F}_1$ is studied in \cite{Na}.
Especially, the period mapping characterizes the tuple $t$ of the parameters.
Namely,
 $t_4, t_6, t_{10},t_{12},t_{18}$ 
 are regarded as  Hermitian modular forms for $SU(2,2)$ (see Theorem \ref{ThmModular}).

Our families $\mathcal{G}_1$ and $\mathcal{F}_1$
are extensions of already known families of $K3$ surfaces as in the 
following diagram.
\begin{align}\label{Diagram}
\xymatrix{
 \mathcal{G}_1 \ar[r]  \ar@{}[d]|{\bigcup}
& \mathcal{F}_1  \ar@{}[d]|{\bigcup} \\
 \mathcal{G}_{{\rm Kum}}  \ar[r] 
& \mathcal{F}_{{\rm CD}}    \\
}
\end{align}
Here, $\mathcal{F}_{\rm CD}$ is the family studied in \cite{CD} (see Section 5.1).
Also, the horizontal arrows are corresponding to the rational double coverings for each member.
We remark that the double covering induces a parametrization of $\mathcal{G}_1$ ($\mathcal{G}_{\rm Kum}$, resp.)
 by Hermitian (Siegel, resp.) modular forms
 derived from the period mapping of $\mathcal{F}_1$ ($\mathcal{F}_{\rm CD}$, resp.).

The first result of this paper is 
to show the precise lattice structure of our family $\mathcal{G}_1$.
We determine the transcendental lattice and the N\'eron-Severi lattice for $\mathcal{G}_1$ (Theorem \ref{ThmTrK(t)}).
In fact,
it is a non-trivial problem to determine them.
The situation is explained by the following facts:
\begin{itemize}

\vspace{-2mm}
\item The lattice structure of $\mathcal{F}_1$ is already known (see Section 2).
However, as mentioned at the beginning of Section 4,
the lattice structure of $\mathcal{G}_1$ is not directly calculated from that of $\mathcal{F}_1$.

\vspace{-2mm}
\item The N\'eron-Severi lattices for families coming from simple $K3$ singularities are studied by Belcastro \cite{B}.
However,
unfortunately,
her result for $\mathcal{G}_1$ is not correct (see Remark \ref{RemB}). 
\end{itemize}
\vspace{-2mm}
In Section 4,
we will give a geometric construction of $2$-cycles on a generic member of $\mathcal{G}_1$,
in order to determine the precise  structure of the transcendental lattice.
Also, by applying arithmetic properties of even lattices,
we will calculate the N\'eron-Severi lattice.

The second result  of this paper is 
to obtain a detailed description of  the double covering between $\mathcal{G}_1$ and $\mathcal{F}_1$.
The double covering between the subfamilies 
$\mathcal{G}_{\rm Kum}$ and $\mathcal{F}_{\rm CD}$
induces an interesting phenomenon called the Kummer sandwich (for detail, see Section 5).
This phenomenon is studied by many researchers (for example, see \cite{S}, \cite{Ma}, \cite{BMS} or \cite{MS}).
We will prove that
we cannot find a complete sandwich phenomenon
 for our families $\mathcal{G}_1$ and $\mathcal{F}_1$ (Theorem \ref{ThmNonExistence}).
In the meanwhile,
the restrictions of our expressions (\ref{K}) and (\ref{S}) to $t_{18}=0$
are useful to consider the Kummer sandwich between $\mathcal{G}_{\rm Kum}$ and $\mathcal{F}_{\rm CD}$ (see Section 5.1).
Furthermore,
we will see a handy description of a sandwich between the family $\mathcal{G}_{\rm MSY}$ of \cite{MSY} and the family $\mathcal{F}_{\rm CMS}$ of \cite{CMS} (see Section 5.3).

In Table 1,
we show the properties of our pair $(\mathcal{G}_1,\mathcal{F}_1)$
compared with those of other pairs which are already known.
Here, ${\rm Tr}$ means the transcendental lattice of a generic member.  
Also, $U$ is the unimodular hyperbolic lattice of rank $2$.
Moreover, for a lattice $\mathcal{L}$ with the intersection matrix $(c_{j,k})_{j,k}$ and $n\in \mathbb{Z}$,
we denote by $\mathcal{L}(n)$ the lattice given by the intersection matrix $n(c_{j,k})_{j,k}$.
\begin{table}[h]
\center
\begin{tabular}{cccccc}
\toprule
$(\mathcal{G},\mathcal{F})$ & $(\mathcal{G}_1,\mathcal{F}_1)$  & $(\mathcal{G}_{{\rm Kum}},\mathcal{F}_{{\rm CD}})$ &  $(\mathcal{G}_{{\rm MSY}}, \mathcal{F}_{{\rm CMS}})$  \\
 \midrule
${\rm Tr}$ for $\mathcal{G}$ 
     &  $U(2)^{\oplus 2}\oplus A_2(-2)$  &  $ U(2)^{\oplus 2}\oplus A_1(-2)$   & $ U(2)^{\oplus 2}\oplus A_1(-1) ^{\oplus 2} $   \vspace{1mm} \\
${\rm Tr}$ for $\mathcal{F}$ 
  &  $ U^{\oplus 2}\oplus A_2(-1)$ &   $ U^{\oplus 2}\oplus A_1(-1)$   & $ U^{\oplus 2} \oplus A_1(-1) ^{\oplus 2}$  \vspace{1mm} \\
Weights of parameters &  $4,6,10,12,18$  & $4,6,10,12$   &  $4,6,8,10,12$ &  \\
Sandwich &  only one side & both sides    & both sides   \\
 \bottomrule
\end{tabular}
\caption{Comparison of our pair $(\mathcal{G}_1,\mathcal{F}_1)$ with other pairs  }
\end{table}

Our family $\mathcal{G}_1$ must be one of the most reasonable extension of $\mathcal{G}_{\rm Kum}$.
It has several good properties.
The moduli space of $\mathcal{G}_1$ is a natural extension of that of the family of principally polarized Abelian surfaces.
Also, the Hermitian modular forms $t_j$ $(j\in \{4,6,10,12,18\})$ have  exact expression using theta functions on the symmetric domain for $SU(2,2)$ and invariants of a complex reflection group (see \cite{NS}).
Furthermore, it seems to the authors that  
the results of the recent work \cite{ILP},
in which motives of $K3$ surfaces are studied,
 are closely related to our surfaces  $K(t)$ and $S(t)$.
Hence, 
the authors expect that those properties of our families provide new merits for future research of arithmetic theory of $K3$ surfaces.

  \section{ $K3$ surfaces from  singularities}
 
 \subsection{Simple $K3$ singularities}
 
 In this subsection, we survey the concept and properties of simple $K3$ singularities.
 For detail, see \cite{IW}, \cite{W}, \cite{Y} and \cite{B}.

 Let $(\mathcal{X},x)$ be a normal isolated singularity in an analytic space $\mathcal{X}$ of three-dimension.
 Let $\rho:(\widetilde{\mathcal{X}},E) \rightarrow (\mathcal{X},x)$ be a good resolution, 
 where $E$ is the exceptional set.
 For any positive integer $m$,
 the plurigenera $\delta_m(\mathcal{X},x)$ is defined as
 ${\rm dim}_\mathbb{C} ( \Gamma(\mathcal{X}-\{x\},\mathcal{O}(mK^\circ))/L^{\frac{2}{m}} (\mathcal{X}-\{x\}))$.
 Here, 
 $K^\circ $ is the canonical bundle on $\mathcal{X}-\{x\}$
 and
 $L^{\frac{2}{m}}(\mathcal{X}-\{x\})$ is the set of holomorphic $m$-ple $3$-forms on $\mathcal{X}-\{x\}$ which are $L^{\frac{2}{m}}$-integrable at $x$.
 A singularity $(\mathcal{X},x)$ is said to be purely elliptic 
 if it holds $\delta_m(\mathcal{X},x)=1$ for any $m$.
 
 Suppose that $(\mathcal{X},x)$ is quasi-Gorenstein.
 Letting $E=\cup_i E_i$ be the irreducible decomposition,
 we have an expression in the form
 $
 K_{\widetilde{\mathcal{X}}}=\pi^* K_\mathcal{X} +\sum_{i\in I} m_i E_i -\sum_{j\in J} m_j E_j,
 $
 with
$m_i, m_j  \in \mathbb{Z}_{>0}$.
If $(\mathcal{X},x)$ is purely elliptic, 
we can prove that $m_j=1$ holds for any $j\in J$.
Now, setting $E_J=\sum_{j\in J} E_j$,
there is an unique $j\in\{0,1,2\}$ such that
$H^{2}(E_J,\mathcal{O}_J) \simeq H^{0,j}(E_J) \simeq \mathbb{C}$.
Such a singularity $(\mathcal{X},x)$ is said to be of $(0,j)$-type.
Now, we have the following result.

\begin{prop} (\cite{IW} Section III)
For a $3$-dimensional isolated singularity $(\mathcal{X},x)$,
the following two conditions are equivalent.

\noindent
(i) $(\mathcal{X},x)$ is Gorenstein, purely elliptic and of $(0,2)$-type,

\noindent
(ii) $(\mathcal{X},x)$ is quasi-Gorenstein and the exceptional divisor $E$ for any minimal resolution $\rho: (\widetilde{\mathcal{X}},E)\rightarrow (\mathcal{X},x)$ is a normal $K3$ surface. 
\end{prop}

If $(\mathcal{X},x)$ satisfies the conditions of the above proposition,
it is called a simple $K3$ singularity.

In the following,
we assume that 
a $K3$ simple singularity $(\mathcal{X},x)$ is given by a hypersurface singularity.
Let   $\mathcal{X}$ be a hypersurface 
$
\mathcal{X}=\{(\zeta_1,\zeta_2,\zeta_3,\zeta_4)\in \mathbb{C}^4| F(\zeta_1,\zeta_2,\zeta_3,\zeta_4)=0\}
$
where $F$ is a non-degenerate polynomial given by
\begin{align} \label{HyperSurface}
F(\zeta_1,\zeta_2,\zeta_3,\zeta_4)=\sum_{(p_1,p_2,p_3,p_4)} \lambda_{p_1,p_2,p_3,p_4} \zeta_1^{p_1} \zeta_2^{p_2} \zeta_3^{p_3} \zeta_4^{p_4}=\sum_p \lambda_p \zeta^p.
\end{align}
Also, let
 $x$ be the origin of the $(\zeta_1,\zeta_2,\zeta_3,\zeta_4)$-space.
Here,
the points $p=(p_1,p_2,p_3,p_4)\in \mathbb{Z}_{\geq 0} ^4$ with $\lambda_p\not =0$
 are integral points of a $3$-dimensional Newton polytope $\Delta$. 
 Such a polytope $\Delta$ is determined by a weight vector 
 $\mathfrak{a}=(\mathfrak{a}_1,\mathfrak{a}_2,\mathfrak{a}_3,\mathfrak{a}_4)\in \mathbb{Q}_{>0}^4$ as follows.
 Let $\Delta'$ be the convex hull of
 $$
 \Big\{\Big(\frac{1}{\mathfrak{a}_1},0,0,0\Big),\Big(0,\frac{1}{\mathfrak{a}_2},0,0\Big),\Big(0,0,\frac{1}{\mathfrak{a}_3},0\Big),\Big(0,0,0,\frac{1}{\mathfrak{a}_4}\Big)\Big\}.
 $$
 Then, $\Delta$ is the convex hull of all integral points of $\Delta'$.
We can prove that
 $(\mathcal{X},x)$ is a simple $K3$ singularity
 if and only if
 the corresponding $3$-dimensional polytope $\Delta$ contains the point $(1,1,1,1)$ in the relative interior (see \cite{W}).
 We note that such a weight vector $\mathfrak{a}$ satisfies
 $\sum_{i=1}^4 \mathfrak{a}_i p_i=1$
 for every integral point $p\in \Delta$.
 Especially,
 it holds
 $\sum_{i=1}^4 \mathfrak{a}_i =1$. 
 Yonemura \cite{Y} classifies such weight vectors into $95 $ types.

 We can obtain  minimal resolutions of the above simple $K3$ singularities as follows. 
 Let  the weight vector $\mathfrak{a}$ be given by $(\frac{w_1}{w},\frac{w_2}{w},\frac{w_3}{w},\frac{w_4}{w})$ such that ${\rm gcd}(w_i,w_j,w_k)=1$ for  distinct  $i,j,k\in \{1,2,3,4\}$.
 Then, from the $4$ unit vectors in $\mathbb{Z}^4_{\geq 0}$ and the integral vector $(w_1,w_2,w_3,w_4),$
 we can obtain a $4$-dimensional fan $\Upsilon \subset \mathbb{R}_{\geq 0}^4$.
So, we have a $4$-dimensional toric variety $V_\Upsilon$.
 There is a canonical morphism $\tilde{\rho}: V_\Upsilon \rightarrow \mathbb{C}^4$ such that 
 $V_\Upsilon-\tilde{\rho}^{-1}(0)$ is isomorphic to $\mathbb{C}^4-\{0\}$
 and
 $\tilde{\rho}^{-1}(0)$ is a $3$-dimensional weighted projective space $\mathbb{P}(w_1,w_2,w_3,w_4).$
 Letting $\widetilde{\mathcal{X}}$ be the proper transform of $\mathcal{X}$ by $\tilde{\rho}$,
 set $\rho=\tilde{\rho}|_{\widetilde{\mathcal{X}}}$.
  Then,  we have the following result.
  
  \begin{prop}(\cite{Y} Section 3, see also \cite{IW} Section IV)
  The morphism
  $\rho: (\widetilde{\mathcal{X}},E)\rightarrow (\mathcal{X},x)$ gives a minimal resolution of  $(\mathcal{X},x)$.
  Here, the exceptional divisor
  is given by a $2$-dimensional set
  $\rho^{-1}(0)=\tilde{\rho}^{-1}(0) \cap \widetilde{\mathcal{X}}$,
  which is a hypersurface  in $\mathbb{P}(w_1,w_2,w_3,w_4)$ corresponding to (\ref{HyperSurface}).
  \end{prop}

Let $\mathfrak{K}_\lambda$ be the hypersurface in $\mathbb{P}(w_1,w_2,w_3,w_4)$ in the above proposition.
 So,
we obtain a family $\{\mathfrak{K}_\lambda\}_\lambda$ of $K3$ hypersurfaces in $\mathbb{P}(w_1,w_2,w_3,w_4)$ with 
complex parameters $\lambda=\lambda_{p_1,p_2,p_3,p_4}$.
In fact,  each weight vector of the 95 vectors of \cite{Y} corresponds to  one of the ``famous 95'' families of $K3$ weighted projective hypersurfaces with Gorenstein singularities 
 due to Reid (for example, see \cite{R} Section 4). 
 Moreover,
 Belcastro \cite{B} 
 studies the lattice structure of a generic member of each family of $95$ families
 from the viewpoint of mirror symmetry of $K3$ surfaces.

 \subsection{Family of $K3$ surfaces corresponding to a weight vector}
 
 In this paper,
 we  focus on the weight vector
 $$
 \mathfrak{a}=\Big( \frac{11}{27},\frac{1}{3}, \frac{5}{27} ,\frac{2}{27} \Big).
 $$
 This is the No.88 in Yonemura's list of the weight vectors (see \cite{Y}).
 The purpose of the present paper is
 to show that the family coming from the  singularity of No.88
gives a natural extension of the family of Kummer surfaces. 
 
 In this case, 
 the polynomial $F(\zeta)$ of (\ref{HyperSurface}) is given by a linear combination of $11$ monomials
 $$
 \zeta_1^2 \zeta_3, \zeta_1 \zeta_4^8, \zeta_2^3, \zeta_2 \zeta_4^9, \zeta_3^5 \zeta_4, \zeta_3 \zeta_4^{11}, \zeta_1 \zeta_2 \zeta_3 \zeta_4, \zeta_2^2 \zeta_3 \zeta_4^2, \zeta_1 \zeta_3^2 \zeta_4^3, \zeta_2 \zeta_3^2 \zeta_4^4, \zeta_3 \zeta_4^{11}
 $$
 over $\mathbb{C}.$
 We can regard $F(\zeta)$ as a quasi homogeneous polynomial of weight $27$,
where  the weight of $\zeta_1$ ($ \zeta_2, \zeta_3, \zeta_4$, resp.) is $11$ ($9,5,2$, resp.).
 A $K3$ surface $\mathfrak{K}_\lambda$ is corresponding to the hypersurface $\{F(\zeta)=0\}$
 in $\mathbb{P}(11,9,5,2).$
 So,
  by putting $\zeta_1=x_0, \zeta_2=y_0, \zeta_3=z_0$ and $\zeta_4=1$,
  $\mathfrak{K}_\lambda$ is given by a hypersurface in $\mathbb{C}^3$,
 which is defined by a linear combination of monomials
 \begin{align}\label{monomialsxyz}
 x_0^2 z_0, x_0, y_0^3, y_0, z_0^5, z_0, x_0 y_0 z_0, y_0^2 z_0, x_0 z_0^2, y_0 z_0^2, z_0.
 \end{align}
A generic member of this family of complex surfaces is a $K3$ surface.
 By an appropriate birational transformation,
 the defining equation of  a generic member can be transformed to an equation 
 \begin{align}\label{E88}
  \hat{K}(t) :\quad  x_1^2 = y_1^3 +(t_4 z^4 + t_{10} z^2) y_1+(z^8 + t_6 z^6 + t_{12} z^4 + t_{18} z^2).
 \end{align}
 We note that the equation (\ref{E88}) defines an elliptic surface 
with $5$ complex parameters  $t_4, t_6, t_{10},t_{12}$ and $t_{18}$.

  \begin{rem}
  Let $\mathcal{R}(x_0,y_0,z_0)=0$ be a linear combination of the monomials in (\ref{monomialsxyz}).
  Then, one can find  
  appropriate constants $c_y, c_z \in \mathbb{C}$
and   polynomial $\psi(y_1',z)\in \mathbb{C}[y_1',z]$
  such that
  $\mathcal{R}(x_0,y_0,z_0)=0$ is transformed to
  \begin{align}\label{EqCss}
  x_1^2= y_1'^3 + (t_2' z) y_1'^2 +(t_4'z^4+ t_{10}' z^2) y_1'+(z^8+ t_6' z^6 + t_{12}' z^4+ t_{18}' z^2)
  \end{align}
 by the birational transformation $(x_0,y_0,z_0) \mapsto (x_1,y_1',z)$
with  $x_0=\frac{x_1}{z^2}+\frac{\psi(y_1',z)}{z}, y_0 =c_y \frac{y_1'}{z}, z_0=c_z z$.
Here, $t_j'$  ($j\in\{2,4,6,10,12,18\}$) are complex numbers.
The form (\ref{EqCss}) can be birationally transformed to (\ref{E88}).
  \end{rem}

 \section{Partner surfaces and Hermitian modular forms}
 
 In order to investigate the $K3 $ surfaces $\hat{K}(t)$,
 let us consider  the partner surfaces $\hat{S}(t)$ of  (\ref{S(t)}).
 We can consider the moduli of the family of $\hat{S}(t)$,
 since 
 they have  clear structures of marked lattice polarized $K3$ surfaces as in \cite{Na} Section 2. 
   In this section, we survey those results.
   For detailed proofs, see \cite{Na} and \cite{NS}.

 Let $t=(t_4,t_6,t_{10},t_{12},t_{18})\in \mathbb{C}^5-\{0\}.$
 Let us consider
  the elliptic surfaces
 \begin{align}\label{S(t)}
  \hat{S}(t):\quad  z_2^2=y_2^3 + (t_4 x_2^4  + t_{10} x_2^3 )y_2 + (x_2^7 + t_6 x_2^6  + t_{12} x_2^5  + t_{18} x_2^4 ).
 \end{align}
 Suppose the weight of $x_2$ ($y_2,z_2, t_j$, resp.) is $6$ ($14, 21,j$, resp.),
 the equation of (\ref{S(t)}) is quasi homogeneous of weight $42$.
So, we have the family of the complex surfaces $\hat{S}(t)$ over
 $
  \mathbb{P}(4,6,10,12,18)={\rm Proj}(\mathbb{C}[t_4,t_6,t_{10},t_{12},t_{18}]).
 $ 
The point of $\mathbb{P}(4,6,10,12,18)$ corresponding to $t=(t_4,t_6,t_{10},t_{12},t_{18}) \in \mathbb{C}^5 -\{0\}$
 is denoted by
 $[t]=(t_4:t_6:t_{10}:t_{12}:t_{18})$.
 Set
 \begin{align}\label{TPara}
 T= \mathbb{P}(4,6,10,12,18) - \{[t] \in \mathbb{P}(4,6,10,12,18) \hspace{0.5mm} | \hspace{0.5mm} t_{10}=t_{12}=t_{18}=0 \}.
\end{align}
If $[t] \in T$, then the complex surface (\ref{S(t)}) is a $K3$ surface.

  The $K3$ lattice $L_{K3}$, which is isomorphic to the $2$-homology group of a $K3$ surface, is isometric to
 $II_{3,19}=U\oplus U\oplus U \oplus E_8(-1) \oplus E_8(-1).$
 Also, the lattice
 \begin{align}\label{LatticeA}
 A=U \oplus U \oplus A_2(-1)
 \end{align}
 of signature $(2,4)$ is necessary for our study.

 \begin{prop}\label{PropPartnerLattice}
  (1) (\cite{Na} Corollary 1.1) \label{NSTrProp}
 For a generic point $[t]\in T$,
 the transcendental lattice (the N\'eron-Severi lattice, resp.) of
 the $K3$ surface $\hat{S}(t)$ of (\ref{S(t)})
 is given by the  intersection matrix
 $A$ of (\ref{LatticeA}) ($U\oplus E_8(-1) \oplus E_6(-1)$, resp.).
 
 (2) (\cite{Na} Proposition 2.3)
  For a generic point $[t]\in \{t_{18}=0\}\cap T$,
 the transcendental lattice (the N\'eron-Severi lattice, resp.) of
 $\hat{S}(t)$
 is given by the  intersection matrix
 $U\oplus U\oplus A_1(-1)$  ($U\oplus E_8(-1) \oplus E_7(-1)$, resp.).
Such a surface coincides with the $K3$ surfaces studied in \cite{CD}.
 \end{prop}

We remark that $A$ is the simplest lattice satisfying the Kneser conditions,
 which are  arithmetic conditions for quadratic forms.
 Due to the Kneser conditions,
 a discrete group derived from $A$ has good properties  as follows.
 First, we have a $4$-dimensional space
 $
\mathcal{D}_M= \{ [\xi] \in \mathbb{P} (A \otimes \mathbb{C}) \hspace{0.5mm} | \hspace{0.5mm}      (\xi, \xi) =0,  (\xi, \overline{\xi}) >0\}
$
from $A$.
Let $\mathcal{D}$ be a connected component of $\mathcal{D}_M$.
We have a subgroup
 $
 O^+(A) = \{\gamma \in O(A) \hspace{0.5mm} | \hspace{0.5mm} \gamma(\mathcal{D}) = \mathcal{D}\}.
 $
Also, set
 $
  \tilde{O}(A) = {\rm Ker}(O(A) \rightarrow {\rm Aut}(A^\vee/A)),
 $
 where $A^\vee ={\rm Hom}(A,\mathbb{Z})$.
 Set 
 \begin{align}\label{DefGamma}
 \Gamma = \tilde{O}^+(A) = O^+(A) \cap \tilde{O}(A).
 \end{align}
Due to the Kneser conditions of $A$,
we can prove that
 $\Gamma = \tilde{O}^+ (A)$ is generated by the reflections of $(-2)$-vectors
 and ${\rm Char}(\Gamma)={\rm Hom}(\Gamma, \mathbb{C}^\times)$ is equal to $\{{\rm id}, {\rm det}\}$.

 We can define the multivalued period mapping $\Phi: T\rightarrow \mathcal{D}$  defined by
 \begin{align}\label{PerPhi}
  [t] \mapsto \Big(\int_{\Delta_{1,t}} \omega_{[t]}^S : \cdots: \int_{\Delta_{6,t}} \omega_{[t]}^S \Big),
 \end{align}
 where $\omega_{[t]}^S$ is the unique holomorphic $2$-form up to a constant factor and $\Delta_{1,t},\ldots, \Delta_{6,t}$ are appropriate $2$-cycles on $\hat{S}(t)$.
 We note that the cycles $\Delta_{1,t},\ldots, \Delta_{6,t}$ are constructed geometrically and explicitly  in \cite{NS} Section 2.
We can prove that this mapping  induces the biholomorphic isomorphism 
\begin{align}\label{PhiIso}
\bar{\Phi} : T \simeq \mathcal{D}/\Gamma
\end{align}
by a detailed argument of the period mappings for lattice polarized $K3$ surfaces.
Especially,
(\ref{PhiIso}) means that
$T$ of (\ref{TPara}) gives the moduli space of marked pseudo-ample $(U\oplus E_8(-1) \oplus E_6(-1))$-polarized $K3$ surfaces
(for detail, see \cite{Na} Section 2).
Thus, 
we have a clear expression of the moduli space of our partner surface (\ref{S(t)}).

The period mapping for the family of $\hat{S}(t)$ is closely related to modular forms.
Let $\mathcal{D}^*$ be the $\mathbb{C}^\times$-bundle of $\mathcal{D}$.
If a holomorphic function $f:\mathcal{D}^* \rightarrow \mathbb{C}$ 
given by $Z\mapsto f(Z)$ satisfies the conditions
\begin{itemize}

\item[(i)] $f(\lambda Z) = \lambda^{-k} f(Z) \quad (\text{for all } \lambda \in \mathbb{C}^*)$,

\item[(ii)] $f(\gamma Z) = \chi(\gamma) f(Z) \quad (\text{for all } \gamma \in \Gamma),$

\end{itemize}
where $k\in \mathbb{Z}$ and $\chi \in {\rm Char}(\Gamma)$,
then $f$ is called a modular form of weight $k$ and character $\chi$ for the group $\Gamma$.

\begin{thm}\label{ThmModular} (\cite{Na}, Theorem 5.1)
(1)  
The ring  $\mathcal{A}(\Gamma,{\rm id})$ of modular forms of  character {\rm id} is isomorphic to
the ring  $\mathbb{C}[t_4,t_6,t_{10},t_{12},t_{18}]$.
Namely, 
via the inverse of the period mapping $\bar{\Phi}$ of (\ref{PhiIso}),
the correspondence $Z\mapsto t_k(Z)$ gives a modular form of weight $k$ and character ${\rm id}$.

(2) There is a modular form $s_{54}$  of weight $54$ and character ${\rm det}$.
Here, $s_{54}$ is given by $s_9 s_{45}$, where
$s_9$ and $s_{45}$ are holomorphic functions on $\mathcal{D}^*$ such that
\begin{align*}
s_9^2 =t_{18},  \quad \quad
s_{45}^2 =(\text{a polynomial in } t_4,t_6,t_{10},t_{12},t_{18} \text{of weight } 90 \text{ in \cite{Na} Section 1}).
\end{align*}
These relations determine the structure of the ring $\mathcal{A}(\Gamma)$ of modular forms with characters.
\end{thm}

There is a biholomorphic mapping between $\mathcal{D}$ and the $4$-dimensional  complex bounded symmetric domain $\mathbb{H}_I$ of type $I$.
Via this  biholomorphic mapping, 
the modular forms in Theorem \ref{ThmModular} are identified with the Hermitian modular forms for the imaginary quadratic field of the smallest discriminant.
Moreover, 
we have an explicit expression of the Hermitian modular forms in Theorem \ref{ThmModular}
by theta functions introduced by Dern-Krieg \cite{DK} (see \cite{NS}).

Thus, 
the family of $K3$ surfaces $\hat{S}(t)$ is very closely related to Hermitian modular forms and theta functions.
Such a relation is a natural and non-trivial counterpart of a classical relation among the Weierstrass form of elliptic curves, elliptic modular forms and Jacobi theta functions.

 \section{Lattice theory for double coverings for $K3$ surfaces}
 
 \subsection{Lattice theoretic properties}

 Let $L$ be a non-degenerate even lattice.
 Let $(s_+,s_-)$ be the signature of $L$.
 We have a natural embedding
 $L\hookrightarrow L^\vee ={\rm Hom}(L,\mathbb{Z})$.
 We set $\mathscr{A}_L = L^\vee/L$.
 The length $l (\mathscr{A}_L)$ of $\mathscr{A}_L$ is the minimum number of generators of $\mathscr{A}_L$.
 
 Let $q$ be the quadratic form which defines the lattice $L$.
 Then, it induces the discriminant form $q_L : \mathscr{A}_L \rightarrow \mathbb{Q}/2\mathbb{Z}.$
 We have $q_{L_1 \oplus L_2} \simeq q_{L_1} \oplus q_{L_2}.$

 \begin{prop}\label{PropLatticeUnique}
 (\cite{NikQuad} Corollary 1.13.3, \cite{Mo} Theorem 2.2)
 Let $L$ be an even lattice with the conditions
 $s_+>0, s_->0$ and $l(\mathscr{A}_L)\leq {\rm rank}(L)-2$.
 Then, $L$ is the unique lattice with invariants $(s_+,s_-,q_L)$ up to isometry.
 \end{prop}

 \begin{prop}\label{PropLatticePerp}
 (\cite{NikQuad} Proposition 1.6.1, \cite{Mo} Lemma 2.4)
 For a unimodular lattice $L$ and a primitive embedding $M\hookrightarrow L$,
 it holds $q_{M^\perp} \simeq - q_M$.
 
 Conversely,
 if $M_1$ and $M_2$ are non-degenerate even lattices satisfying $q_{M_1} \simeq -q_{M_2}$,
 then there are a unimodular lattice $L$ and a primitive embedding $M_1 \hookrightarrow L$
 such that $M_1^\perp \simeq M_2$.
 \end{prop}

Next,  let us summarize  results of Nikulin involutions.
 For detail,
 see \cite{Ni1}, \cite{Ni2} and \cite{Ma}.
  Let $\mathfrak{X}$ be an algebraic $K3$ surface
 and
  $\omega$ be the unique holomorphic $2$-form on $\mathfrak{X}$ up to a constant factor.
 An involution $\iota \in {\rm Aut}(\mathfrak{X})$ is called a Nikulin involution (or symplectic involution),
 if it holds $\iota^* \omega =\omega$.
 If $\iota$ is a Nikulin involution on $\mathfrak{X}$,
 then the minimal resolution $\mathfrak{Y}=\widetilde{\mathfrak{X}/\langle \iota \rangle}$ is also an algebraic $K3$ surface.
 Namely, we have a rational quotient mapping $\mathfrak{X} \dashrightarrow \mathfrak{Y}.$
 Conversely,
 any given rational double covering $\mathfrak{X}\dasharrow \mathfrak{Y}$ of $K3$ surfaces is derived from   a Nikulin involution.

 Let $\iota$ be a Nikulin involution on $\mathfrak{X}$
 and set 
 $\mathfrak{T}^\iota = H_2(\mathfrak{X},\mathbb{Z})^{\iota_*}$.
 Then, we can see that the transcendental lattice ${\rm Tr}(\mathfrak{X})$ is a primitive sublattice of $\mathfrak{T}^\iota$.
 The orthogonal complement $(\mathfrak{T}^\iota)^\perp$ of $\mathfrak{T}^\iota$ in $L_{K3}$
 has remarkable properties.
 The lattice $(\mathfrak{T}^\iota)^\perp$ is an even negative definite lattice and has no $(-2)$-vectors
 (see \cite{Ni1} Lemma 4.2).
 Also, this is a lattice of
  rank $8$ and  determinant $2^8$.
  Moreover, we have $((\mathfrak{T}^\iota)^\perp)^\vee / (\mathfrak{T}^\iota)^\perp \simeq (\mathbb{Z}/2\mathbb{Z})^8$  (see \cite{Ni1} Proposition 10.1).
 According to these properties, 
 we can see that  $(\mathfrak{T}^\iota)^\perp $ is isometric to $E_8(-2)$ and 
 \begin{align} 
  \mathfrak{T}^\iota \simeq U\oplus U\oplus U\oplus E_8(-2).
 \end{align}
 Thus, the lattice $E_8(-2)$ is important to study Nikulin involutions.
 We will use the following result.

 \begin{lem}(\cite{Ni2} Section 2, see also \cite{Ma} Proposition 2.1) \label{LemNikulin}
 Let $\mathfrak{X}$ be an algebraic $K3$ surface.

 (1) If $\mathfrak{X}$ admits a Nikulin involution $\iota$, there is a primitive embedding  
 \begin{align}\label{embTr}
 {\rm Tr} (\mathfrak{X}) \hookrightarrow U\oplus U\oplus U\oplus E_8(-2).
 \end{align}
Also, one has the structure of  the transcendental lattice of $\mathfrak{Y}=\widetilde{\mathfrak{X}/\langle \iota \rangle}$:
\begin{align}\label{TransViaD}
{\rm Tr}(\mathfrak{Y}) \simeq \Big(({\rm Tr}(\mathfrak{X})\otimes \mathbb{Q}) \cap \Big(U\oplus U\oplus U \oplus \frac{1}{2} E_8(-2)\Big)\Big)(2)
\end{align}

(2) Conversely, if the transcendental lattice of $\mathfrak{X}$ admits a primitive embedding (\ref{embTr}),
there is a Nikulin involution $\iota$. 

 \end{lem}

 \subsection{Double covering for our families}

 In this subsection, let us give a relation between the $K3$ surface $\hat{K} (t)$ of (\ref{E88}), which is coming from the simple $K3$ singularity of No.88, and  the partner surface $\hat{S}(t)$ of (\ref{S(t)}), which is closely related to Hermitian modular forms and the simplest lattice $A$ of (\ref{LatticeA}) with the Kneser conditions.
 
 \begin{thm}\label{ThmKS}
 Suppose $[t]$ be a point  $T$ of (\ref{TPara}).
 
 (1) The $K3$ surface $\hat{K} (t)$ of (\ref{E88}) is birationally equivalent to the surface
 \begin{align} \label{KE} 
 K(t): Z^2=Y^3 +\Big( t_4 + \frac{t_{10}}{U^2} \Big)Y +\Big( t_6 + U^2+ \frac{ t_{12}}{U^2}  +\frac{t_{18}}{U^4} \Big).
 \end{align}
 
 (2)
 The $K3$ surface $\hat{S}(t)$ of (\ref{S(t)}) is birationally equivalent to the surface
 \begin{align} \label{SE} 
 S(t): Z^2=Y^3 + \Big(t_4 +\frac{t_{10}}{X}\Big) Y + \Big(  t_6 +X+ \frac{ t_{12}}{X} +\frac{t_{18}}{X^2}\Big). 
 \end{align}
 
 (3) One has an explicit double covering $K(t) \dashrightarrow S(t)$ given by $(U,Y,Z) \mapsto (X,Y,Z) = (U^2,Y,Z)$,
 which is coming from a Nikulin involution $\iota$ on $K(t)$. 
 \end{thm}

 \begin{proof}
 (1) By putting
 $$
 x_1=U^3 Z ,\quad \quad y_1=U^2 Y , \quad \quad z=U,
 $$
 the equation (\ref{E88}) is transformed to (\ref{KE}).
 
 (2) By putting
 $$
 x_2=X, \quad\quad y_2=X^2 Y, \quad \quad z_2= X^3 Z,
 $$
 the equation (\ref{S(t)}) is transformed to (\ref{SE}).
 
 (3) 
 Since the unique  holomorphic $2$-form on $K(t)$ is given by
 $\frac{dY \wedge dU}{UZ}$ up to a constant factor,
 an explicit involution given by
$(U,Y,Z)\mapsto (-U,Y,Z)$
 is a Nikulin involution  on $K(t)$.
 \end{proof}

 The family $\pi_{\mathcal{G}_1}:\mathcal{G}_1\rightarrow T$ ($\pi_{\mathcal{F}_1}:\mathcal{F}_1\rightarrow T$, resp.)  in Introduction is equal to the family of $K(t)$ ($S(t)$, resp.).

  \begin{rem}\label{RemPartner}
 The reason why $S(t)$ is named the partner surface is the following.
The subfamily $\pi^{-1}_{\mathcal{F}_1} \left(\{t_{18}=0\} \cap T \right)$ 
 is equal to the family $\mathcal{F}_{\rm CD}$ of $K3$ surfaces studied in \cite{CD} (see Proposition \ref{PropPartnerLattice} (2) and Section 5.1).
Since every member of $\mathcal{F}_{{\rm CD}}$ has a van Geemen-Sarti involution (see \cite{GS} Section 4),
it is called a van Geemen-Sarti partner of the Kummer surface (for example, see \cite{CMS}).
However, as we will see in Section 5.2,
a generic member of our family $\mathcal{F}_1$  admits no van Geemen-Sarti involutions.
Hence, we simply call  $\mathcal{F}_1$ the family of partner surfaces of $\mathcal{G}_1$.
 \end{rem}

 \begin{rem}
 The reason why we use the expressions (\ref{KE}) and (\ref{SE}) with denominators is the following.
Firstly, the double covering $K(t) \dasharrow S(t)$ is described in  a very simple form by using (\ref{KE}) and (\ref{SE}). 
Also, our expressions are natural extensions of the results of \cite{S},
in which he explicitly studies a subfamily $\mathcal{G}_{{\rm Shio}}$ of $\mathcal{G}_{{\rm Kum}}$ consisting of the Kummer surfaces attached to direct products of two elliptic curves 
and a subfamily $\mathcal{F}_{{\rm Shio}}$ of $\mathcal{F}_{{\rm CD}}$ corresponding to $\mathcal{G}_{{\rm Shio}}$ (for detail, see the beginning of Section 5).
 \end{rem}

 \begin{rem}
 The authors found the very simple expressions (\ref{KE}) and (\ref{SE}) of $K3$ surfaces
 during our research 
 in order to describe our period mappings via solutions of
a system of GKZ hypergeometric differential equations. 
This investigation will be published  elsewhere.
 \end{rem}

  \begin{prop}\label{PropKPic}
 For a generic point $[t]=(t_4:t_6:t_{10}:t_{12}:t_{18}) \in T$,
 the Picard number of the $K3$ surface $K(t)$ of (\ref{KE}) is equal to $16$.
 \end{prop}

 \begin{proof}
 By Proposition \ref{NSTrProp} and Lemma \ref{LemNikulin} (especially (\ref{TransViaD})),
 the rank of ${\rm Tr}(K(t))$ is equal to that of ${\rm Tr}(S(t))$.
 From Proposition \ref{NSTrProp}, we have the assertion.
 \end{proof}

 \section{Geometric construction  of  transcendental lattice}

In this section,
we will determine the lattice structure of the  surface $K(t)$ of (\ref{KE}).

According to Proposition \ref{PropKPic},
the  rank of the transcendental lattice of $K(t)$ is generically equal to $6$.
If we had a double covering $S(t) \dashrightarrow K(t), $
we could determine the transcendental lattice of $K(t)$
by applying Lemma  \ref{LemNikulin} to Proposition \ref{NSTrProp}.
However, in practice,
$S(t)$ does not have any Nikulin involutions,
as we will see in Section 5.2.
Therefore,
in this section, we will construct
$2$-cycles on $K(t)$ geometrically,
in order to
determine the lattice structure of $K(t)$ directly.
Eventually, we will prove the following theorem.

 \begin{thm}\label{ThmTrK(t)}
 (1) For a generic point $[t]=(t_4:t_6:t_{10}:t_{12}:t_{18}) \in T$,
 the transcendental lattice
  of the $K3$ surface $K(t)$ of (\ref{KE})
  is given by
 the intersection matrix $A(2)=U(2)\oplus U(2) \oplus A_2(-2)$. 
 This is a primitive sublattice of $ U\oplus U \oplus U \oplus E_8(-2).$
 
 (2)
 For a generic point $[t] \in T$, 
 the N\'eron-Severi lattice of $K(t)$ is isometric to the lattice $U(2) \oplus E_8(-1) \oplus L_6$,
 where
 \begin{align}\label{M6}
 L_6=
 \begin{pmatrix} 
 -2 & 1 & 1 & 0 & 0 & 0\\
 1&-4&0&2&0&0\\
1 &0&-4&0&2&0\\
 0&2&0&-4&2&0\\
 0&0&2&2&-4&2\\
 0&0&0&0&2&-4\\
 \end{pmatrix}.
 \end{align}
 \end{thm}
 
 \begin{rem}\label{RemB}
 Belcastro \cite{B}
 determines the structures of the lattices for the $K3$ surfaces which are coming from resolutions of simple $K3$ singularities.
 Especially, in the list of \cite{B},
 it is stated that the transcendental lattice corresponding to the $K3$ surfaces for the singularity of No.88
  is isomorphic to $A$ of (\ref{LatticeA}).
 If that is true,
 the family of $K(t)$ should coincide with that of $S(t)$.
 
 We tried to follow such a story, but it was not successful.
 On the contrary,
 we found that the transcendental lattice of the family of $K(t)$ should be $A(2)$
 and this family naturally contains the family of the algebraic Kummer surfaces. 
 
 Thus,
 at least for the case of No.88,
 the statement of \cite{B} is not correct.
 Our Theorem \ref{ThmTrK(t)} gives a correction of this defect.
 \end{rem}

In order to calculate the transcendental lattice of $K(t)$,
 we consider the elliptic $K3$ surface $\hat{K} (t)$ of (\ref{E88}), which is birationally equivalent to $K(t)$.
There is an involution 
\begin{eqnarray} \label{eq: frakj}
\jmath : (x_1,y_1,z)\mapsto  (x_1,y_1,-z).
\end{eqnarray}
Also, on $\hat{K} (t)$,  there is  the unique  holomorphic 2-form 
\begin{eqnarray} \label{eq: holoformS0}
\omega_t=\omega = \frac{dy_1\wedge dz}{x_1}
\end{eqnarray}
up to a constant factor.
Then, we have
$
\jmath^{\ast} (\omega) =-\omega.
$
Namely,
$\jmath$
is not a Nikulin involution on the $K3$ surface (\ref{E88}).
The quotient surface
\begin{align}\label{Sigma(t)}
\Sigma(t):\quad  x_1^2=y_1^3 + ( t_4 w^2 +t_{10} w) y_1 +(w^4 +t_6 w^3 +t_{12} w^2 +t_{18} w) 
\end{align}
 by the involution $\jmath$ is a  rational surface.

\begin{rem}
We remark that the period mapping for our family $\mathcal{F}_1$ is connected to the complex reflection group of No.33  in the list of \cite{ST}.
For detail, see \cite{NS}.

On the other hand, Sekiguchi \cite{Sek} studies the complex reflection groups of No.33 from the viewpoint of Frobenius potentials
and obtains a family of rational surfaces given by the Weierstrass equation
 $$
 z^2=f_{E_6(1)} =y^3 + (  t_2 x^2 + t_5 x) y + (x^4 + t_3 x^3 + t_6 x^2 + t_9 x)  + s_4 y^2.
 $$
 If  $s_4=0$, 
 we have our rational surface $\Sigma(t)$ of (\ref{Sigma(t)}).
 Although our method and viewpoint are very different from those of Sekiguchi,
 our surface $\Sigma(t)$ is very close to his surface.
\end{rem}

By putting $(t_4,t_6,t_{10},t_{12},t_{18})=(0,0,3,-1,-1)$,
we take a generic member
\begin{eqnarray} \label{eq:K0}
&&
\hat{K}_0:\quad x_1^2=y_1^3+3z^2y_1+z^8-z^4-z^2
\end{eqnarray}
of the family of the  surfaces of (\ref{E88}).
Namely, we regard $\hat{K}_0$ as a reference surface of the family.
By virtue of a consideration of local period mapping,
as in the argument in \cite{Na} Section 1.3,
we can see that
the lattice structure for the reference surface $\hat{K}_0$ is valid for generic members of the family of $K(t)$ of (\ref{KE}). 
We  obtain the corresponding quotient surface $\Sigma_0$ of $\hat{K}_0$ by  $\jmath$:
\begin{eqnarray}  \label{eq:surfaceSigma0}
&&
\Sigma_0:\quad x_1^2=y_1^3+3w y_1+w^4-w^2-w.
\end{eqnarray}
 Properties of $\Sigma_0$ will be useful to construct  $2$-cycles on $\hat{K}_0$.

\subsection{Rational surface $\Sigma_0$}

Let us observe the elliptic surface $(\Sigma_0, \pi , \mathbb{P}^1(\mathbb{C}))$ over $w$-sphere $\mathbb{P}^1(\mathbb{C})$.
Let $\tau_1, \tau_2, \tau_3, \tau_4, \tau_5$ and $\tau_6$ be the solutions of 
$
1 + 6 w + w^2 - 2 w^3 - 2 w^4 + w^6 = 0
$
such that approximate values are 
 $
  (\tau_1,\tau_2,\tau_3,\tau_4,\tau_5,\tau_6)\fallingdotseq (-1.40553,-0.173094,-0.662487 + 1.14871 i, -0.662487 - 1.14871 i, 1.4518 + 0.479369 i,1.4518 - 0.479369 i).
  $ 
 We have singular fibres of $\Sigma_0$ of type $I_1$   over these points.
 Moreover, there is a singular fiber of type $II$ ($IV$, resp.) over 
 $w=\tau_0=0$ ($w=\tau_{\infty}=\infty$, resp.).
 Set $\tau_b=-0.5$ be a base point on the $w$-space.
We have a regular fiber
$$
\pi^{-1}(\tau_b): x_1^2=y_1^3   - 1.5 y_1 + 0.3125.
$$
It has four ramification points on $y_1$-plane:
$
(v_1,v_2,v_3,v_m)\fallingdotseq (-1.31799,0.214955,1.10304, \infty).
$
On $\pi^{-1}(\tau_b)$, set an oriented closed arc $\gamma_1$ ($\gamma_2$, resp.) such that its projection goes around 
$v_1$ and $v_2$ ($v_2$ and $v_3$, resp.) in the positive direction. 
We define their branch and orientation so that the intersection number $(\gamma_1\cdot \gamma_2)$ is equal to $1$ 
(see Figure 1).
\begin{figure}[h]
\center
\includegraphics[scale=0.7]{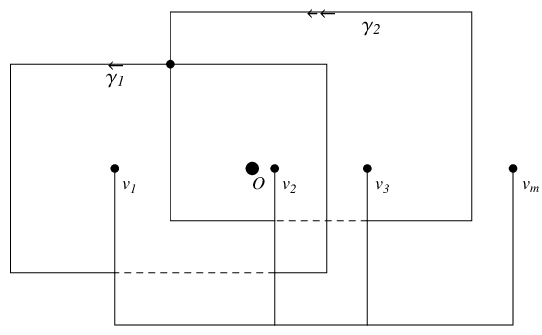}
\caption{Cycles $\gamma_1, \gamma_2$ on $\pi^{-1}(\tau_b)$}
\end{figure}

We make an oriented closed circuits $\delta_i$ $(i=0,1,\ldots, 6, \infty )$ on the $w$-plane which goes around $\tau_i $ in the positive 
direction with the starting point  $\tau_b$  (see Figure 2).

\begin{figure}[h]
\center
\includegraphics[scale=0.8]{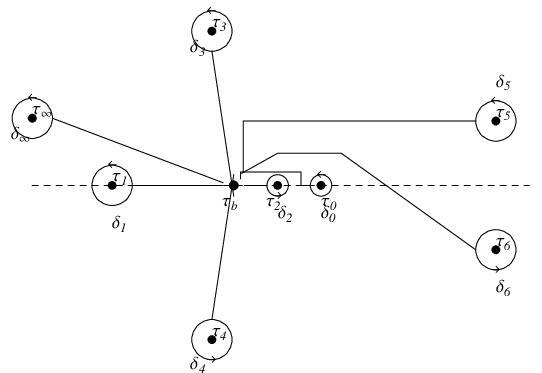}
\caption{Circuits $\delta_i (i=0,1,\ldots, 6, \infty )$}
\end{figure}

Let $M_i$ be the matrix  which represents the monodromy transformation of the homology basis ${}^t( \gamma_1, \gamma_2)$ 
by the continuation of the fiber $\pi^{-1}(\tau_{b})$ along $\delta_i$. 
We call it the circuit matrix. 
They are given in Table 2.
 Note that they are assumed to be left actions.
\begin{table}[h]
\center
{\small
\begin{tabular}{ccccccccc} \hline
&$\tau_i$&$\tau_1$&$\tau_2$&$\tau_3$&$\tau_4$&$\tau_5$&$\tau_6$&\\ \hline
&matrix&$\begin{pmatrix}1&-1\cr0&1\end{pmatrix}$&$\begin{pmatrix}1&-1\cr0&1\end{pmatrix}$&$\begin{pmatrix}1&0\cr1&1\end{pmatrix}$&
$\begin{pmatrix}1&0\cr1&1\end{pmatrix}$&$\begin{pmatrix}1&-1\cr0&1\end{pmatrix}$&$\begin{pmatrix}1&0\cr1&1\end{pmatrix}$&
\\ \hline
&V. cycle&$\gamma_2$&$\gamma_2$&$\gamma_1$&$\gamma_1$&$\gamma_2$&$\gamma_1$
\\ \hline
\end{tabular}
\vskip 1mm
\begin{tabular}{ccccccccc} \hline
&$\tau_i$&$\tau_0$&$\tau_{\infty}$&\\ \hline
&matrix&$\begin{pmatrix}1&-1\cr1&0\end{pmatrix}$&$\begin{pmatrix}-1&-1\cr1&0\end{pmatrix}$&
\\ \hline
&V. cycle&any cycle&any cycle
\\ \hline
\end{tabular}}
\caption{Circuit matrix of $\delta_i$}
\end{table}
Here, ``V. cycle" means the 1-cycle which vanishes at $\tau_i$.
We have
$
M_1M_4M_2M_0M_6M_5M_3M_{\infty} =\rm{id},
$
since
$\delta_1\delta_4\delta_2\delta_0\delta_6\delta_5\delta_3\delta_{\infty} $
is homotopic to  zero.

\subsection{Construction of  cycles on  $\hat{K}_0$}
We have the elliptic $K3$ surface $(\hat{K}_0, \pi_{z}, \mathbb{P}^1(\mathbb{C}))$ of (\ref{eq:K0}),
where the base space $\mathbb{P}^1(\mathbb{C})$ is the $z$-sphere and $\pi_{z}$ is the natural projection.
There is
the double covering  $\hat{K}_0 \rightarrow \Sigma_0$ given by the involution (\ref{eq: frakj}).
Recalling the singular fibers over $\Sigma_0$,
we find singular fibers of the elliptic $K3$ surface $\hat{K}_0$ over the points
$$
\zeta_j = \sqrt{\tau_j} ,\quad  \zeta_j'=-\sqrt{\tau_j} \quad(j\in \{1,\ldots ,6\}),\quad \quad
\zeta_0=0 , \quad 
\zeta_{\infty} =\infty.
$$

\begin{rem} \label{remsingulartypesandLB}
We have a singular fiber of type  $IV$ ($IV^*$, resp.)  over $\zeta_0$ ($\zeta_\infty$, resp.).
So, besides the general fiber and the global section, 
we have two (six, resp.) components 
of $\pi_{z}^{-1}(0)$ ($\pi_{z}^{-1}(\infty)$, resp.). 
Hence, we obtain a sublattice of ${\rm NS}(\hat{K}_0)$ of rank 10 which is generated by these divisors.
We denote it by $L_B$.
\end{rem}

For $i\in\{1,\cdots,0,\infty\}$,
let $\ell_i$ be a line segment
connecting a fixed point $\underline{B}$ in the  $w$-sphere
and $\tau_i$ respectively.
We lift up the cut lines  $\ell_1,\ldots, \ell_6,\ell_0,\ell_{\infty}$,
 to 
 $m_1,\ldots, m_6, m_1',\ldots , m_6', m_0,m_{\infty}$. 
 Those are indicated in Figure 3. 
We take an oriented arc $\alpha$ from $\zeta_b=\sqrt{\tau_b}$ to  $\zeta_b'=-\sqrt{\tau_b}$ in the simply connected region 
$\mathbb{P}'= \mathbb{P}^1(\mathbb{C}) - (\bigcup_i m_i \cup \bigcup_i m_i' \cup {\tilde{BB'}})$, where $B$ ($B'$, resp.) is the initial points of 
the cut lines $m_1, m_2, m_4, m_3',m_5',m_6'$ and $m_0$  ($m_3, m_5,m_6,m_1',m_2',m_4'$ and $m_\infty$, resp.) and ${\tilde{BB'}}$ is an arc connecting them indicated in Figure 3.
 \begin{figure}[h]
\center
\includegraphics[scale=1.0]{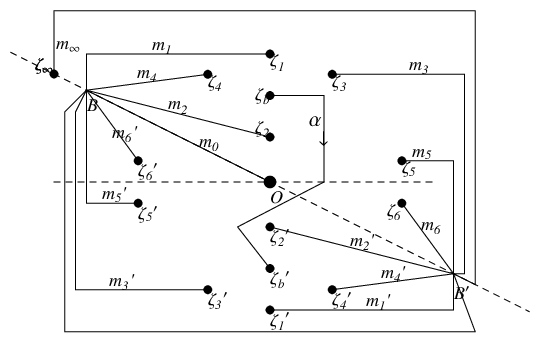}
\caption{A simply connected region $\mathbb{P}'$}
\end{figure}

We make the liftings $\delta_{zi}$ $(i= 1,\ldots,6)$ of $\delta_i$ indicated in 
Figure 4. 
In any case,
 we take $\zeta_b$ as their starting point.
By a similar manner, we make circuits $\delta_{zi}'$ starting from $\zeta_b'$ as indicated in Figure 4.
Also, we take a closed circle $\delta_{z0}$ ($\delta_{z\infty}$, resp.) around $z=0$ ($z=\infty$, resp.).
 \begin{figure}[h]
\center
\includegraphics[scale=0.84]{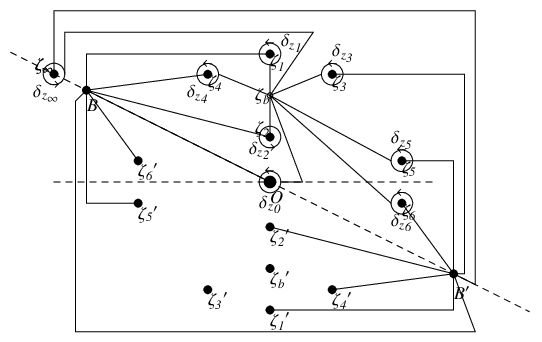}
\includegraphics[scale=0.84]{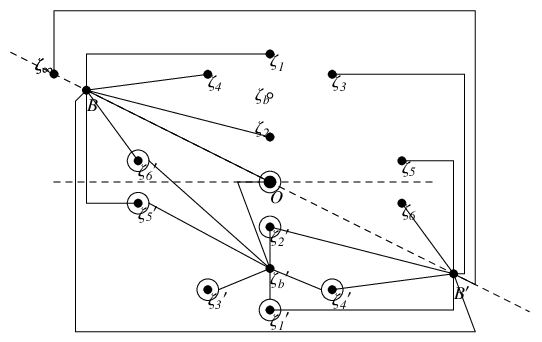}
\caption{Circuits $\delta_{zi}$ and $\delta_{zi}'$}
\end{figure}

Then,  $\pi_z^{-1}(\zeta_b)$ and $\pi_z^{-1}(\zeta_b')$ are the same elliptic curve which is identified with
$\pi^{-1}(\tau_b)$ on $\Sigma_0$. 
So we can define $\gamma_1$ and $\gamma_2$  on them by this identification.
On the other hand,  the involution $\jmath$ in ({\ref{eq: frakj}}) induces an isomorphism 
$
\pi_z^{-1}(\zeta_b) \simeq \pi_z^{-1}(\zeta_b')
$
on $\hat{K}_0$. 
So,  we obtain the $1$-cycles
$
\gamma_1'=\jmath (\gamma_1)$ and $\gamma_2' =\jmath (\gamma_2)
$
on $\pi_z^{-1}(\zeta_b')$ .

We have the local monodromy of the system $\{ \gamma_1,\gamma_2\}$
along every circuit $\delta_{zi}$. We denote the matrix which represents it by $M_{zi}$. 
By observing the covering structure $\hat{K}_0 \rightarrow \Sigma_0$ and Figure 2,
we have $M_{zi} =M_i \ (i=1,\ldots, 6)$ and $M_{z0} =M_0^2$.

\begin{table}[h]
\center
{\small
\begin{tabular}{ccccccccc} \hline
&$\delta_{zi}$&$\delta_{z1}$&$\delta_{z2}$&$\delta_{z3}$&$\delta_{z4}$&$\delta_{z5}$&$\delta_{z6}$&\\ \hline
&matrix $M_{zi}$ &$\begin{pmatrix}1&-1\cr0&1\end{pmatrix}$&$\begin{pmatrix}1&-1\cr0&1\end{pmatrix}$&$\begin{pmatrix}1&0\cr1&1\end{pmatrix}$&
$\begin{pmatrix}1&0\cr1&1\end{pmatrix}$&$\begin{pmatrix}1&-1\cr0&1\end{pmatrix}$&$\begin{pmatrix}1&0\cr1&1\end{pmatrix}$&
\\ \hline
&V. cycle&$\gamma_2$&$\gamma_2$&$\gamma_1$&$\gamma_1$&$\gamma_2$&$\gamma_1$
\\ \hline
\end{tabular}

\vspace{2mm}
\begin{tabular}{ccccccccc} \hline
&$\delta_{zi}$&$\delta_{z0}$&$\delta_{z\infty}$&\\ \hline
&matrix $M_{zi}$&$\begin{pmatrix}0&-1\cr1&-1\end{pmatrix}$&$\begin{pmatrix}0&1\cr -1&-1\end{pmatrix}$&
\\ \hline
&V. cycle&any cycle&any cycle
\\ \hline
\end{tabular}
}
\caption{Circuit matrix $M_{zi}$}
\end{table}

For the local monodromy induced from $\delta_{zi}'$ on the system $\{ \gamma_1', \gamma_2'\}$
 is just the copy of those in Table 3, we have the same matrix for $i\in \{1,\ldots ,6\}$.
We have the transformation between two systems $\{ \gamma_1,\gamma_2\}$ and $\{ \gamma_1',\gamma_2'\}$ 
on $\pi_z^{-1} (\zeta_b')$ induced from the arc $\alpha$:
\begin{eqnarray} \label{eq: gammadashtransform}
&&
\begin{pmatrix} \gamma_1'\cr \gamma_2' \end{pmatrix} 
=M_{\alpha} \begin{pmatrix} \gamma_1\cr \gamma_2\end{pmatrix} ,\quad\quad
\text{where } M_\alpha=\begin{pmatrix} 1&-1\cr 1&0\end{pmatrix}  .
\end{eqnarray}
Then, the local monodromy $M_{zi}'$ induced from $\delta_{zi}'$ on the system $\{ \gamma_1, \gamma_2\}$
is given by
$
M_{zi}'=M_{\alpha}^{-1} M_{zi} M_{\alpha}.
$
Hence, we obtain Table 4 for them.
\begin{table}[h]
\center
{\small
\begin{tabular}{ccccccccc} \hline
&$\delta_{zi}'$&$\delta_{z1}'$&$\delta_{z2}'$&$\delta_{z3}'$&$\delta_{z4}'$&$\delta_{z5}'$&$\delta_{z6}'$&\\ \hline
&matrix $M_{zi}'$ &$\begin{pmatrix}1&0\cr1&1\end{pmatrix}$&$\begin{pmatrix}1&0\cr1&1\end{pmatrix}$&$\begin{pmatrix}2&-1\cr1&0\end{pmatrix}$&
$\begin{pmatrix}2&-1\cr1&0\end{pmatrix}$&$\begin{pmatrix}1&0\cr1&1\end{pmatrix}$&$\begin{pmatrix}2&-1\cr1&0\end{pmatrix}$&
\\ \hline
&V. cycle&$\gamma_1$&$\gamma_1$&$\gamma_1-\gamma_2$&$\gamma_1-\gamma_2$&$\gamma_1$&$\gamma_1-\gamma_2$
\\ \hline
\end{tabular}
}
\caption{ Circuit matrix $M_{zi}'$}
\end{table}

\begin{rem}
We can determine $M_{\alpha}$ by the facts that $M_{z0}'=M_{z0}$ and that the fiber $\pi_z^{-1}(0)$ is a singular 
fiber of type IV.
By the relation
$
M_{0z}M_{z6}'M_{z5}'M_{z3}'M_{z1}'M_{z4}'M_{z2}'M_{z6}M_{z5}M_{z3}M_{z\infty}M_{z1}M_{z4}M_{z2} =\rm{id},
$
the matrix $M_{z\infty}$ is determined as in Table 3.
\end{rem}

Let $\rho$ be an oriented arc in the $z$-sphere $\mathbb{P}^1(\mathbb{C})$ from $\xi$ to $\eta$,
where $\xi $ and $\eta$ are two points in $\mathbb{P}^1(\mathbb{C})$.
Take $\gamma$ be a $1$-cycle on the fibre $\pi_z^{-1}(\xi)$.
We let  $U(\rho,\gamma)$ denote the $2$-chain obtained by the analytic continuation of $\gamma$ along $\rho$.
Here, we define the orientation of $U(\rho,\gamma)$ as the ordered pair of the orientation of $\rho$ and that of $\gamma$ as in \cite{MSY} Chapter 2.
We note that
if $\gamma$ is a vanishing cycle at $\xi$ and $\eta$,
then $U(\rho,\gamma)$ becomes to be a $2$-cycle on $\hat{K}_0$.

Let us construct a $2$-cycle $\Gamma_1$.
We take the  points $Q_1,Q_2,Q_3$ and $Q_4$ on the $z$-plane indicated in Figure 5. 
Then, we make an ``$8$-shaped'' 
closed arc $\rho_1$ connecting them in this order returning to the initial point $Q_1$ (see Figure 5). We take a 
$1$-cycle $\gamma_1$ on the fiber $\pi_z^{-1}(Q_1)$. 
Then, 
 we obtain a
 $2$-cycle $\Gamma_1=U(\rho_1,\gamma_1)$.
 According to Table 2, we can see that 
the continuation returns back to the original $\gamma_1$. 
Namely, the cycle $\gamma_1$ changes to $\gamma_1-\gamma_2$ after crossing the cut line $m_1$ and it returns back 
to $\gamma_1$ after crossing $m_2$.
 So $\Gamma_1$ is a $2$-cycle on $\hat{K}_0$.  
\begin{figure}[h]
\center
\includegraphics[scale=0.92]{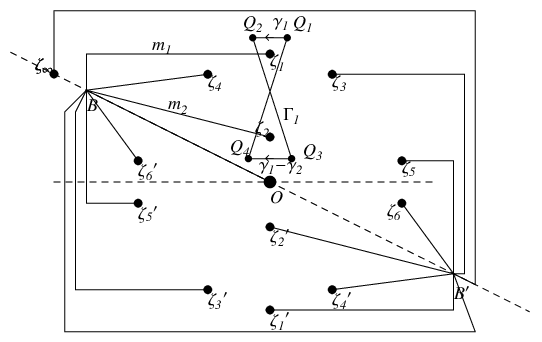}
\caption{Detail of $2$-cycle $\Gamma_1$  on $\hat{K}_0$}
\end{figure}

By the same way, according to the indication in Figure 6, we can obtain  $2$-cycles
$\Gamma_2, \Gamma_3$ and $\Gamma_4$.

\begin{figure}[h]
\center
\includegraphics[scale=0.92]{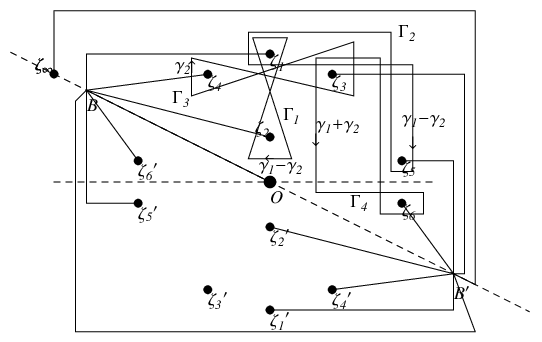}
\caption{Cycles $\Gamma_1,\ldots, \Gamma_4$ on $\hat{K}_0$}
\end{figure}

Next, we construct a $2$-cycle $\Gamma_5$.
We take the points $P_1$ and $P_2$ illustrated in Figure 7.
Take a small circle $c_1$ ($c_2$, resp.),
 which starts from $P_1$ and goes around $\zeta_3$ ($\zeta_5$, resp.) in the positive (negative, resp.) direction.
Let 
$\beta_5$ be an oriented arc from $P_1$ to $P_2$ crossing the cut line $m_0$.
So, we have the $2$-chains
$U(c_1,\gamma_2), U(\beta_5,\gamma_1)$ and $U(c_2, \gamma_1)$ on $\hat{K}_0$.
Here, 
 we use the circuit matrices in Table 3.
For example,
 $\gamma_2$ on $P_1$ is transformed to 
$\gamma_1+\gamma_2$ after crossing the cut line $m_3$.
Thus, we obtain a $2$-cycle
$$
\Gamma_5=U(c_1,\gamma_2)\cup U(\beta_5,\gamma_1)\cup U(c_2, \gamma_1),
$$
that is illustrated by Figure 7.
\begin{figure}[h]
\center
\includegraphics[scale=0.92]{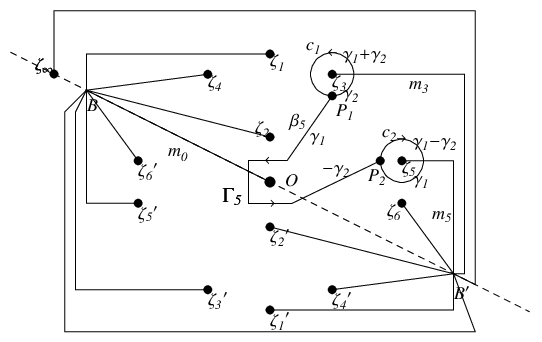}
\caption{Cycle $\Gamma_5$ on $\hat{K}_0$}
\end{figure}

Also, let us construct a $2$-cycle $\Gamma_6$.
Take the points $P_2,P_3,P_4$ and $P_5$ indicated in Figure 8.
Let $c_2$ be as above.
Let $c_3$ ($c_5$, resp.) be a small circle which starts from $P_3$ ($P_5$, resp.) and goes around $\zeta_6$ ($\zeta_6'$, resp.) in the negative (positive, resp.) direction.
For $(l,m)\in \{(4,2),(4,3),(5,4)\}$,
let $\beta_{lm}$ be an arc from $P_l$ to $P_m$.
We have the  $2$-chains 
$U(c_2,\gamma_1-\gamma_2), U(c_3,\gamma_1+\gamma_2), U(c_5,\gamma_2'), 
U(\beta_{42},-\gamma_2), U(\beta_{43},\gamma_1)$ and  $U(\beta_{54}, \gamma_1')$.
Here, according to (\ref{eq: gammadashtransform}),
 we have $\gamma_2'= \gamma_1$ and $\gamma_1'= \gamma_1-\gamma_2$.
As in the case of $\Gamma_5$, 
we obtain a $2$-cycle  
$$
\Gamma_6= U(c_2,\gamma_1-\gamma_2)\cup  U(c_3,\gamma_1+\gamma_2)\cup  U(c_5,\gamma_2')\cup   
U(\beta_{42},-\gamma_2)\cup  U(\beta_{43},\gamma_1)\cup  U(\beta_{54}, \gamma_1'),
$$
which is illustrated  in Figure 8.
\begin{figure}[h]
\center
\includegraphics[scale=0.92]{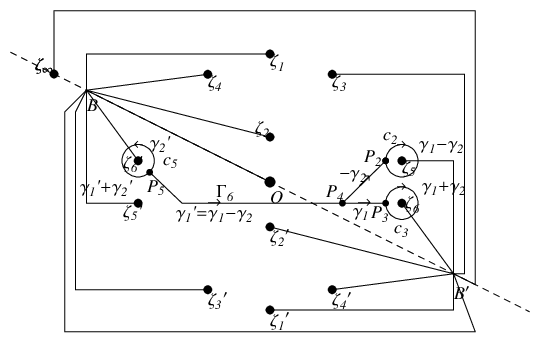}
\caption{Cycle $\Gamma_6$ on $\hat{K}_0$}
\end{figure}

\subsection{Intersection numbers}
We put
$
\Gamma_i'= \jmath (\Gamma_i) \ (i\in \{ 1,2,3,4,5,6\}).
$
Set 
\begin{eqnarray} \label{GammaGammadashsystem}
L_{GG'}=\langle \Gamma_1,\cdots , \Gamma_6, \Gamma_1', \cdots ,\Gamma_6'\rangle.
\end{eqnarray}
\begin{prop}
The rank of  $ L_{GG'}$  is equal to $12$.  
It is orthogonal to
the system $L_B$ in Remark \ref{remsingulartypesandLB}.
Hence, it holds
$
\langle L_{GG'}, L_B\rangle \otimes \mathbb{Q} = H_2(\hat{K}_0,\mathbb{Q}).
$
\end{prop}

\begin{proof}
We have ${\rm rank} (H_2(\hat{K}_0,\mathbb{Q}))=22$ and ${\rm dim}_\mathbb{Q} (L_B\otimes \mathbb{Q})=10$.
By the construction, any member of $L_{GG'}$ is orthogonal to $L_B$.
So, it is enough to check that  $\Gamma_1,\cdots, \Gamma_6, \Gamma_1', \cdots , \Gamma_6'$ 
are independent. 
Since we have   Proposition \ref{GG'intersectionmatrix} below, 
it follows that the intersection matrix to be nonsingular.
Hence we have the assertion.
\end{proof}

\begin{prop} \label{GG'intersectionmatrix}
The intersection matrix $M_{GG'}$ of the system (\ref{GammaGammadashsystem}) is given by 
$
M_{GG'} = \begin{pmatrix} C_G& C_{GG'}\cr {}^tC_{GG'}& C_G
\end{pmatrix},
$
where
\begin{align*}
&  C_G=((\Gamma_i\cdot \Gamma_j))_{1\leq i,j \leq 6 }=((\Gamma_i'\cdot \Gamma_j'))_{1\leq i,j \leq 6 }
=
{\small
\begin{pmatrix}
-2&-1&1&0&0&0\cr
-1&-2&1&1&0&-1\cr
1&1&-2&-1&-1&0\cr
0&1&-1&-2&0&0\cr
0&0&-1&0&-2&-1\cr
0&-1&0&0 &-1&-2
\end{pmatrix}},\\
& C_{GG'} =((\Gamma_i\cdot \Gamma_j'))_{ 1\leq i,j \leq 6 }
=
{\small
\begin{pmatrix}
0&0&0&0&0&0\cr
0&0&0&0&0&0\cr
0&0&0&0&0&0\cr
0&0&0&0&0&1\cr
0&0&0&0&-2&-1\cr
0&0&0&1&-1&0
\end{pmatrix}}.
\end{align*}
Especially, $M_{GG'}$ is nonsingular.
\end{prop}

\begin{proof}
Let us calculate the intersection number $(\Gamma_1\cdot \Gamma_2)$.
We have two geometric intersections $R_1$ and $R_2$ on the $z$-plane as described in Figure 9.
We observe both local intersections $(\Gamma_1\cdot \Gamma_2)_{R_1}$ and $(\Gamma_1\cdot \Gamma_2)_{R_2}$.
At the point $R_1$, 
$\Gamma_1\cap \pi_z^{-1} (R_1)$ is the $1$-cycle $\gamma_1-\gamma_2$ 
and $\Gamma_2 \cap \pi_z^{-1}(R_1)$ is the same 
$1$-cycle $\gamma_1-\gamma_2$. 
So, it holds $(\Gamma_1\cdot \Gamma_2)_{R_1} =0$.
At the point $R_2$, $\Gamma_1 \cap \pi_z^{-1}(R_2)$ is the $1$-cycle $\gamma_1$ and $\Gamma_2\cap \pi_z^{-1}(R_2)$ is the
$1$-cycle $\gamma_1-\gamma_2$.
Hence,
%By denoting $\Gamma_1=U(\rho_1,\gamma), \Gamma_2=U(\rho_2,\gamma')$,
we have
\[
(\Gamma_1\cdot \Gamma_2)_{R_2} 
= (-1) \cdot (\gamma_1 \cdot (\gamma_1-\gamma_2)) \cdot (\rho_1\cdot \rho_2)_{R_2} =-1.
\]
So $(\Gamma_1 \cdot \Gamma_2)=-1$ holds.
We can similarly calculate other intersection numbers. 
\end{proof}

\begin{figure}[h]
\center
\includegraphics[scale=1]{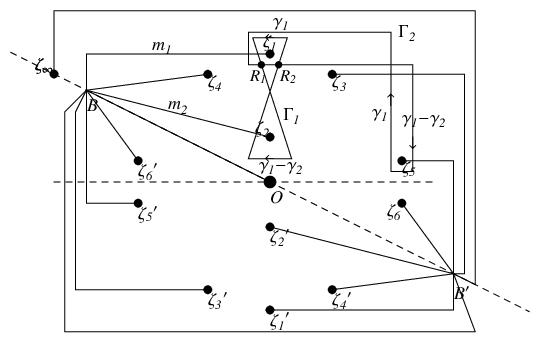}
\caption{Calculation of the intersection number $(\Gamma_1\cdot \Gamma_2)$}
\end{figure}

By (\ref{eq: holoformS0}), we have
$$
\int_{\Gamma_i'} \omega = -\int_{\Gamma_i} \omega  \quad \quad  (i=1,\ldots ,6).
$$
It means that  every $\Gamma_i+\Gamma_i'\ (i=1,\ldots ,6)$ is an algebraic cycle.

\begin{prop}
Every $\Gamma_i-\Gamma_i' \ (i=1,\ldots ,6)$ is an element of the orthogonal complement of the 
N\'eron-Severi lattice ${\rm{NS}} (\hat{K}_0)$.
\end{prop}

\begin{proof}
By  Proposition \ref{PropKPic},
the rank of ${\rm NS}(\hat{K}_0)$ is $16$.
By the construction, it is apparent   that $\Gamma_i-\Gamma_i' $ is orthogonal to the lattice $L_B$. 
Also, we obtain $((\Gamma_i-\Gamma_i' )\cdot (\Gamma_j+\Gamma_j' ))=0$ from Proposition
\ref{GG'intersectionmatrix}.
Hence,
the assertion is proved.
\end{proof}

\begin{figure}[h]
\center
\includegraphics[scale=1]{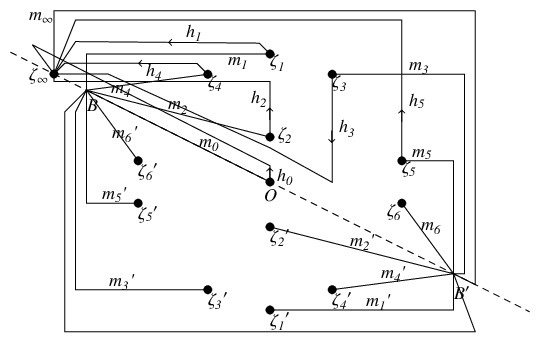}
\caption{ $h$-arcs for dual $2$-cycles}
\end{figure}

We make oriented arcs $h_0,h_1,\ldots, h_5$ on the $z$-plane  given by Figure 10.
 Each arc starts from $\zeta_i$ with 
the terminal $\zeta_{\infty}$. 
Over these arcs, we make  
$2$-cycles:
\begin{align} \label{eq: dualcycles}
\begin{cases}
& C_1=U(h_1, \gamma_2),\quad  C_2=U(h_2, -\gamma_2), \quad  C_3=U(h_3, -\gamma_1), \\
&C_4=U(h_4, \gamma_1), \quad
C_5=U(h_5, -\gamma_2), \quad C_6=U(h_0, -\gamma_2).
\end{cases}
\end{align}
For $C_i \ (i=1,\ldots ,6)$, each $1$-cycle on the fiber on $h_i$  is a vanishing cycle at the starting point $\zeta_i\ 
(i=0,1,\ldots ,5)$. So every $C_i$ determines a $2$-cycle on $\hat{K}_0$.
By a direct calculation as in the proof of Proposition \ref{GG'intersectionmatrix},
we obtain the following result.

\begin{prop}
The intersection matrix $M_{CG}=((C_i\cdot (\Gamma_j-\Gamma_j')))_{i,j \in\{ 1,\ldots ,6\} }$ is given by
\begin{align*}
M_{CG} = 
{\small
\begin{pmatrix}
1&1&0&0&0&0\cr
1&0&0&0&0&0\cr
0&0&1&1&1&0\cr
0&0&1&0&0&0\cr
0&1&0&0&1&1\cr
0&0&0&0&0&1
\end{pmatrix}}.
\end{align*}
It holds ${\rm det} (M_{CG})=1$.
\end{prop}

By this proposition, it is guaranteed that the system $\{ \Gamma_1-\Gamma_1', \ldots , \Gamma_6-\Gamma_6'\}$ becomes to be 
a  $\mathbb{Z}$-basis of  $H_2(\hat{K}_0,\mathbb{Z})/{\rm NS}(\hat{K}_0)$.
Set
$
G_j=\Gamma_j-\Gamma_j'
$
$(j\in \{1,\ldots,6\})$.
Also, setting
$
{}^t(C_1',\cdots ,C_6') = M_{CG}^{-1} {}^t (C_1,\cdots ,C_6),
$
it holds
\begin{eqnarray*}
(C_i'\cdot G_j)=\delta_{ij} \quad\quad (i,j\in \{ 1,\ldots ,6\}).
\end{eqnarray*}
We set $M_G =(G_i\cdot G_j)_{1 \leq i,j \leq 6} .$
By a direct calculation, we have the following proposition.

\begin{prop}
\begin{eqnarray*}
 M_G  =
2
{\small\begin{pmatrix}
-2& -1& 1&0& 0& 0\cr
-1&-2& 1& 1& 0& -1\cr
1& 1& -2& -1& -1& 0\cr
0& 1& -1& -2& 0& -1\cr 
0& 0& -1& 0& 0& 0\cr 
0&-1& 0& -1& 0& -2
\end{pmatrix}}.
\end{eqnarray*}
\end{prop}

\begin{thm}\label{thmGsystem}
The system $\{ G_1, \cdots , G_6\}$ 
gives a basis of the transcendental lattice ${\rm{Tr}} (\hat{K}_0)$
of the reference surface 
with the intersection matrix $A(2)=U(2)\oplus U(2)\oplus A_2(-2)$.
\end{thm}

\begin{proof}
Set
$
M_J={\small
\begin{pmatrix}
0&0&0& 0& 1& 0\cr 
0& 0& -1& 0& 1& 0\cr 
-1& 1& 0& 1& -1& -1\cr 
0&1&0& 1& 0& -1\cr
 0&0& 0& 1& -1& 0\cr
  0&0& 0& 0& 0& -1
\end{pmatrix}}.
$
This is a unimodular matrix.
By a direct calculation, it holds
$
M_J M_G {}^tM_J = \Big(U(2)\oplus U(2)\oplus \begin{pmatrix} -4&2\cr 2&-4 \end{pmatrix}\Big).
$
\end{proof}

Theorem \ref{ThmTrK(t)} (1)  immediately follows
from Theorem \ref{thmGsystem} and Lemma \ref{LemNikulin}.

\subsection{N\'eron-Severi lattice for $\mathcal{G}_1$}

In this subsection, we prove Theorem \ref{ThmTrK(t)} (2).

\begin{lem}\label{LemPrimitive}
(1) There is a primitive embedding $i_1: U(2)\hookrightarrow U\oplus U$.

(2)  There is a primitive embedding $i_2: U(2) \oplus A_2(-2) \hookrightarrow U \oplus E_8(-1)$.
\end{lem}

\begin{proof}
(1) 
If $V$ be an even unimodular lattice of rank $2$,
then there exists a primitive embedding  $V \hookrightarrow U \oplus U$.
The assertion follows from this fact.

(2) Let $\{e,f\}$ be a basis of $U$ similar as above.
 Let $\{p_1,\cdots,p_8\}$ be a basis of $E_8(-1)$ with the Dynkin diagram in Figure 11.
 \begin{figure}[h]
\center
\includegraphics[scale=0.8]{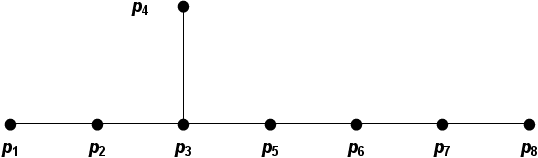}
\caption{Basis of $E_8(-1)$}
\end{figure}
Set $w_1=e+f+p_1, w_2=e+f+p_3, w_3=p_6+p_8$ and $w_4=e-f+p_7$. 
Then,  $\{w_1,w_2,w_3,w_4\}$ gives a basis of  $U(2)\oplus A_2(-2)$ which induces a primitive embedding $U(2) \oplus A_2(-2) \hookrightarrow U \oplus E_8(-1)$.
\end{proof}

\begin{lem}\label{LemOrthogonal}
(1) The intersection matrix of the orthogonal complement of $i_1 (U(2))$ is given by $U(2)$.

(2) The intersection matrix of the orthogonal complement of $i_2 (U(2)\oplus A_2(-2))$ is given by $L_6$ of (\ref{M6}).
\end{lem}

\begin{proof}
(1) is obvious.
We give a proof of (2).
Under the notation of the proof of Lemma \ref{LemPrimitive},
we have a basis $\{b_1,\cdots,b_6\}$ of the orthogonal complement of $\langle w_1,\cdots,w_4 \rangle$, where
$
b_1= e+ f +p_1 +p_3 ,
b_2= -e -p_1 -p_2 -3p_3 -2p_4 -2p_5 -p_6,
b_3= -f +p_2 + p_3 +2p_5 +p_6, 
b_4= -p_1 -2p_2 -p_3,
b_5=-p_3 -2p_5 -2p_6 -p_7
$
and
$
b_6= p_6 -p_8.
$
The intersection matrix  of this basis is $L_6$ of (\ref{M6}).
\end{proof}

For an even lattice $L$,
let $\mathscr{A}_L$ and $q_L$ be  as in Section 3.1.

\begin{lem}\label{LemQDiscr}
(1) The group $\mathscr{A}_{A_2(-1)}$ is isomorphic to $\mathbb{Z}/6\mathbb{Z}$.
The discriminant form $q_{A_2(-1)}$ is a quadratic form corresponding to the intersection matrix $\begin{pmatrix} -1/3 \end{pmatrix}$.

(2) The group $\mathscr{A}_{U(2)}$ is isomorphic to $(\mathbb{Z}/2\mathbb{Z})^2$.
The discriminant form $q_{U(2)}$  is a quadratic form corresponding to the intersection matrix $\begin{pmatrix} 0 & 1/2 \\ 1/2 & 0 \end{pmatrix}$.
\end{lem}

\begin{proof}
(1) Let $\{n_1,n_2\}$ be a basis of $A_2(-2)$ such that 
$(n_1\cdot n_1)=(n_2 \cdot n_2)=-4$ and $(n_1\cdot n_2)=2$.
Then, $\mathscr{A}_{A_2(-2)}=\langle \frac{ n_1 + 2 n_2}{6} \rangle_\mathbb{Z}$.
The form of $q_{A_2(-2)}$ follows immediately.

(2) Let $\{E,F\}$ be a basis of $U(2)$ such that 
$(E\cdot E)=(F \cdot F)=0$ and $(E\cdot F)=2$.
Then, $\mathscr{A}_{U(2)}=\langle \frac{E}{2}, \frac{F}{2} \rangle_\mathbb{Z}$.
The form of $q_{U(2)}$ follows immediately.
\end{proof}

Now, we give a proof of
Theorem \ref{ThmTrK(t)} (2).

\begin{proof}[Proof of Theorem \ref{ThmTrK(t)} (2)]
Letting $\tilde{M}$ be the orthogonal complement of $U(2) \oplus U(2)\oplus A_2(-2)$ in $L_{K3}$,
according to Proposition \ref{PropLatticePerp} and Lemma \ref{LemQDiscr}, 
we have
$$
l(\mathscr{A}_{\tilde{M}})=5 \leq 14 = {\rm rank}(\tilde{M})-2.
$$
From Proposition \ref{PropLatticeUnique}, 
$\tilde{M}$ is uniquely determined by the invariants $(1,15,-(q_{U(2)} \oplus q_{U(2)} \oplus q_{A_2(-2)}))$.

Since we have the primitive embeddings $i_1$ and $i_2$ in Lemma \ref{LemPrimitive},
together with Lemma \ref{LemOrthogonal},
$\tilde{M}$ is regarded as a primitive sublattice of $L'=U\oplus U\oplus U\oplus E_8(-1)$
with the orthogonal complement $U(2) \oplus L_6$.
Also, the orthogonal complement of $L'$ in $L_{K3}$ is $E_8(-1)$.
By virtue of Proposition \ref{PropLatticePerp}, 
the lattice 
$U(2) \oplus E_8(-1) \oplus L_6$ has the invariants $(1,15,-(q_{U(2)} \oplus q_{U(2)} \oplus q_{A_2(-2)}))$.
So, $\tilde{M}$ is isometric to this lattice.
\end{proof}

\subsection{Multivalued period mapping for $\mathcal{G}_1$}

As in \cite{Na} Section 1.3, we can define a marking $H_2(\hat{K} (t), \mathbb{Z}) \rightarrow L_{K3}$ for $[t]\in T$
using a basis of ${\rm NS}(\hat{K} (t))$ and the analytic continuation in the parameter space $T$ of (\ref{TPara}).
Here, the $2$-cycles on the reference surface $\hat{K}_0$ give the initial marking. 
Via such a marking, we can define the period mapping for the family of $\hat{K} (t)$, also.
We can obtain an  expression of the period mapping for $\hat{K} (t)$  under the notation of this section as follows.
Set
$(H_1,\cdots,H_6)=(G_1,\cdots,G_6) {}^t M_J$.
We can expand the system $\{H_1,\cdots, H_6\}$
to $\{H_1,\cdots, H_{22}\}$,
which gives a basis of $H_2(\hat{K}_0,\mathbb{Z}).$
Let $\{D_1,\cdots,D_{22}\}$ be the dual basis of $\{H_1,\cdots,H_{22}\}$
with respect to the unimodular lattice $L_{K3}$.
By using the above mentioned  marking,
we can naturally obtain $2$-cycles $D_{1,t},\ldots, D_{6,t} \in H_2(\hat{K} (t),\mathbb{Z})$.
Now, let us recall 
 the period mapping (\ref{PerPhi}) for the family of the partner surface $S(t)$.
Since it holds ${\rm Tr}(\hat{K} (t)) = {\rm Tr}(S(t)) (2)$,
the Riemann-Hodge relations for $\hat{K}(t)$ are equal to those for $S(t)$.
Especially, the period domain for the family of $\hat{K}(t)$ coincides with $\mathcal{D}$ which appears in Section 2.
Hence, we obtain the following theorem.

\begin{thm}\label{MPThm}
The multivalued period mapping for the family of $\hat{K}(t)$,
which is birationally equivalent to $K(t)$ of (\ref{KE}),
 coincides with (\ref{PerPhi}):
$$
T \ni [t] \mapsto 
\Big(\int_{D_{1,t}}\omega_t:\cdots : \int_{D_{6,t}} \omega_t\Big)=\Big(\int_{\Delta_{1,t}}\omega_{[t]}^S:\cdots:\int_{\Delta_{6,t}}\omega_{[t]}^S\Big) \in \mathcal{D},
$$
where $\omega_t$ is the holomorphic $2$-form of (\ref{eq: holoformS0}) on $\hat{K} (t)$.
\end{thm}

By virtue of the argument at the end of Section 2,
we have the following corollary.

\begin{cor}\label{ThetaCor}
One has an explicit expression of the inverse correspondence of the period mapping for the family of $K(t)$ by the Dern-Krieg theta functions.
The theta expression is the same as that of \cite{NS} Theorem 4.1.
\end{cor}

Thus,
we have an analytic correspondence 
between the period integrals on $K(t) $ and the parameters $t_j$ via the theta functions.

 We note that the family of our explicit surfaces $K(t)$ does not coincide with the set of equivalence classes of marked lattice polarized $K3$ surfaces for the lattice $A(2)^\perp (\subset L_{K3})$ in the sense of \cite{D}.
The moduli space of such polarized $K3$ surfaces is given by a covering of the parameter space $T$ of (\ref{TPara}).
Recall that  $T$   just gives  the moduli space of marked $(U\oplus E_8(-1) \oplus E_6(-1))$-polarized $K3$ surfaces via (\ref{PhiIso}).
 Nevertheless,
 the expression of the multivalued period mapping in Theorem \ref{MPThm} and the theta expression mentioned in Corollary \ref{ThetaCor} are valid.

 \section{Sandwich phenomenon}
 
 Shioda \cite{S} studies the following explicit defining equations of the Kummer surface ${\rm Kum}(E_1\times E_2)$ for the product of two elliptic curves and a $K3$ surface $S_{{\rm Shio}}$:
 \begin{align}\label{ShiodaEq1}
{\rm Kum}(E_1\times E_2) :&\quad Z^2 = Y^3 -3 \alpha Y +\Big(U^2 +\frac{1}{U^2} -2\beta \Big), \\ \label{ShiodaEq2}
S_{{\rm Shio}} :&\quad Z^2 = Y^3 -3 \alpha Y +\Big(X+\frac{1}{X} -2\beta \Big) ,
\end{align}
where these equations define elliptic surfaces and $\alpha$ and $\beta$ are complex parameters.
Let $\mathcal{G}_{{\rm Shio}}$  ($\mathcal{F}_{{\rm Shio}}$, resp.) be the family of all ${\rm Kum} (E_1\times E_2)$ ($S_{\rm Shio}$, resp.). 
 He shows that there are double coverings $\varphi_{{\rm Shio}}$ and $\psi_{{\rm Shio}}$ such that
 \begin{align*}
 {\rm Kum}(E_1\times E_2) \overset{\varphi_{{\rm Shio}}}{\dashrightarrow} S_{{\rm Shio}} \overset{\psi_{{\rm Shio}}}{\dashrightarrow} {\rm Kum}(E_1\times E_2).
 \end{align*}
 Explicitly, $\varphi_{{\rm Shio}}$ and $\psi_{{\rm Shio}}$ are derived from  the involutions
 \begin{align}\label{ShiodaInvolution}
 (U,Y,Z)\mapsto (-U,Y,Z) , \quad \quad (X,Y,Z) \mapsto \Big( \frac{1}{X}, Y, -Z \Big),
 \end{align}
 respectively.
 He calls this phenomenon the Kummer sandwich.
 We note that
 studies for the Kummer sandwich  are  active 
 and
 there are several recent works 
 (for example, see \cite{BMS} or \cite{MS}).
 
 The surface $K(t)$ of  (\ref{KE}) ($S(t)$ of (\ref{SE}), resp.) can be regarded as a natural extension of ${\rm Kum}(E_1\times E_2)$ ($S_{{\rm Shio}}$, resp.).
 In this section, we will study natural and explicit counterparts of the Kummer sandwich phenomenon.

 \subsection{Explicit expression of Kummer sandwich of general type}

 Following the work \cite{S},
 Ma \cite{Ma} proves that
 an arbitrary Kummer surface ${\rm Kum}(\mathfrak{A})$,
 where $\mathfrak{A}$ is a principally polarized Abelian surface,
 admits the Kummer sandwich.
 Namely,
 there are double coverings $\varphi_{{\rm Kum}} $ and $\psi_{{\rm Kum}}$ such that
 \begin{align}\label{MaKummer}
   {\rm Kum}(\mathfrak{A}) \overset{\varphi_{{\rm Kum}}}{\dashrightarrow} S_{{\rm CD}} \overset{\psi_{{\rm Kum}}}{\dashrightarrow} {\rm Kum}(\mathfrak{A})
  \end{align}
  (see \cite{Ma} Theorem 2.5).
  Here, $S_{{\rm CD}}$ is a lattice polarized $K3$ surface with the transcendental lattice $U\oplus U\oplus A_1(-1)$.
The family of  such  $K3$ surfaces coincides with the family studied in \cite{CD} from the viewpoint of the Siegel modular forms.
  This is the reason why we use the notation  $S_{{\rm CD}}$.
  The family $\mathcal{G}_{{\rm Kum}}$ ($\mathcal{F}_{{\rm CD}}$, resp.) in Introduction is the family of ${\rm Kum}(\mathfrak{A})$ ($S_{{\rm CD}}$, resp.)
We note that the proof of \cite{Ma} is based on a lattice theoretic argument,
 but he does not give explicit forms of the defining equations and the double coverings like (\ref{ShiodaEq1}) (\ref{ShiodaEq2}) and (\ref{ShiodaInvolution}).

  As an application of Theorem \ref{ThmKS}, 
 let us give a simple  expression of $\varphi_{{\rm Kum}} $ and $\psi_{{\rm Kum}}$ in (\ref{MaKummer}).
 This explicit result gives a natural extension of (\ref{ShiodaInvolution})  of \cite{S}.

 \begin{thm}\label{ThmMaSand}
 The Kummer surface ${\rm Kum}(\mathfrak{A}) $ for a principally polarized Abelian surface $\mathfrak{A}$ is given by the Weierstrass equation
 \begin{align}\label{Kum(A)}
 Z^2=Y^3 +\Big( t_4  + \frac{t_{10}}{U^2} \Big)Y +\Big( t_6 + U^2+ \frac{ t_{12}}{U^2} \Big). 
 \end{align}
Also, the $K3$ surface 
 $S_{{\rm CD}}$ is given by the Weierstrass equation
 \begin{align}\label{SCD}
 Z^2=Y^3 + \Big(t_4+\frac{t_{10}}{X} \Big)Y +\Big( t_6 + X+ \frac{ t_{12}}{X} \Big).  
 \end{align}
 The mappings $\varphi_{{\rm Kum}}$ and $\psi_{{\rm Kum}}$, which give the Kummer sandwich (\ref{MaKummer}), are explicitly given by 
 the Nikulin involutions
 \begin{align*}
 \iota_{\varphi,{\rm Kum}}: (U,Y,Z)\mapsto (-U,Y,Z),  \quad \iota_{\psi,{\rm Kum}}: (X,Y,Z)\mapsto \Big(\frac{t_{10} Y + t_{12}}{X}, Y, -Z\Big),
 \end{align*}
 respectively.
  \end{thm}

 \begin{proof}
 In \cite{Na}, the $K3$ surface (\ref{S(t)}) with the transcendental lattice $U\oplus U\oplus A_2(-1)$ is  introduced as an extension of the $K3$ surface $S_{{\rm CD}}$ with the transcendental lattice $U\oplus U\oplus A_1(-1)$ (see also Proposition \ref{PropPartnerLattice} (2)).
 Namely, if $t_{18}=0$,  then (\ref{S(t)})  degenerates to $S_{{\rm CD}}$.
 So, together with Theorem \ref{ThmKS}, the surface defined by (\ref{SCD}) is birationally equivalent to $S_{{\rm CD}}$.
 
 The involution $\iota_{\psi,{\rm Kum}}$  defines a Nikulin involution on (\ref{SCD}).
 Since $S_{{\rm CD}}$ admits a Shioda-Inose structure in the sense of \cite{Mo}, 
 the minimal resolution of the quotient $S_{{\rm CD}}/\langle \iota_{\psi,{\rm Kum}}\rangle $ coincides with  ${\rm Kum}(\mathfrak{A})$.
 This image is defined by the equation
 \begin{align}\label{EqU0V0}
 U_0 V_0^2 - U_0^2+ 4 t_{12}  + t_6 V_0^2 +  4 t_{10} Y + t_4 V_0^2 Y + V_0^2 Y^3=0,
 \end{align}
 where $U_0=X+\frac{t_{10} Y + t_{12}}{X}$ and $V_0=\frac{1}{Z} (X - \frac{t_{10}Y+t_{12} }{X})$.
 By the birational transformation $(U_0,V_0)=(2U(U+Z),2U)$, (\ref{EqU0V0}) is transformed to (\ref{Kum(A)}) with the transcendental lattice $U(2)\oplus U(2) \oplus A_1(-2)$
 by virtue of \cite{Mo} Theorem 5.7.
 Therefore, the equation (\ref{Kum(A)}) defines a general algebraic Kummer surface.
 
  The involution $\iota_{\varphi,{\rm Kum}}$ is just a special case of the involution of Theorem \ref{ThmKS}.
 \end{proof}

 \begin{rem}
 The involution $\iota_{\psi,{\rm Kum}}$ is equal to the van Geemen-Sarti involution (see \cite{GS} Section 4)
 for the family of $S_{{\rm CD}}$ studied in \cite{CD}.
 \end{rem}

 Theorem \ref{ThmMaSand} also gives a natural visualization of the Kummer sandwich phenomenon  and Siegel modular forms of degree $2$.
 Recall that the ring of Siegel modular forms on the $3$-dimensional Siegel upper half plane $\mathfrak{S}_2$  of the trivial character is generated by the modular forms of weight $4,6,10,12$ and $35$. 
 In fact,  each of the parameters $t_j$ $(j\in \{4,6,10,12\})$  in Theorem \ref{ThmMaSand} gives a member of a system of  generators   of the ring of Siegel modular forms  via the period mapping for the family of $S_{{\rm CD}}$.
 Also, the generator of weight $35$ and modular forms of the non-trivial character can be calculated  by considering the degeneration of $S_{{\rm CD}}$ (see \cite{CD}, see also \cite{Na} Proposition 2.3).

 \subsection{Nonexistence of sandwich between $\mathcal{G}_1$ and $\mathcal{F}_1$}
 
The family of $K(t)$ ($S(t)$, resp.)  is a natural extension of that of ${\rm Kum}(\mathfrak{A})$ ($S_{{\rm CD}}$, resp.).  
 The argument in the last subsection guarantees that there is the sandwich when $t_{18}=0$.
 However, if $t_{18}\not=0$, 
 we have only one side 
 $$
   K(t) \dashrightarrow S(t)
 $$
 of the sandwich,
 because Theorem \ref{ThmKS} and the following theorem hold.
 
 \begin{thm}\label{ThmNonExistence}
 The parameter $t_{18}$ gives an obstruction to the existence of a double covering  $S(t)\dashrightarrow K(t)$ for the $K3$ surfaces (\ref{KE}) and (\ref{SE}).
  Namely, if $t_{18}\not =0$,
  there is no Nikulin involution on $S(t)$. 
 \end{thm}
 
 \begin{proof}
 In order to apply the results of Section 3.1,
 we assume that there is a primitive embedding
 \begin{align} \label{iPrimEmb}
 i : U \oplus U\oplus A_2(-1) \hookrightarrow U \oplus U\oplus U\oplus E_8(-2).
 \end{align}
 Now,
 this direct summand $E_8(-2)$ should be a primitive sublattice of $E_8(-1) \oplus E_8(-1) (\subset L_{K3})$:
 \begin{align}\label{E8(-2)}
 E_8(-2) \hookrightarrow E_8(-1) \oplus E_8(-1).
 \end{align}
So, we have a primitive embedding 
 \begin{align*}
 i': U \oplus U\oplus A_2(-1) \hookrightarrow L_{K3}=II_{3,19}
 \end{align*}
 as the composition of (\ref{iPrimEmb}) and (\ref{E8(-2)}).
 The orthogonal complement of $U\oplus U$ in $II_{3,19}$ is isometric to $U\oplus E_8(-1)\oplus E_8(-1)$,
 because it is an even unimodular lattice of signature $(1,17)$.
Hence, 
 we can regard  the direct summand $U\oplus U$ of $L_{K3}$ 
 as the image of $i'|_{U\oplus U}$.
 Thus, 
 we can suppose that the embedding $i$ satisfies 
 \begin{align}\label{A2UE8}
 i(A_2(-1)) \subset U \oplus E_8(-2).
 \end{align}
 Now, let $\{e,f\}$ be a  basis of $U$ such that
 $(e\cdot e)=(f\cdot f)=0$ and $(e\cdot f)=1$.
 Also, let $\{\nu_1,\nu_2\}$ be a  basis of $A_2(-1)$ satisfying $(\nu_1\cdot \nu_1)=(\nu_2 \cdot \nu_2)=-2$ and $(\nu_1 \cdot \nu_2)=1$.
If (\ref{A2UE8}) holds,
we have an expression
$$
i (\nu_j)=l_j e +m_j f + \mu_j
$$
 for $j=1,2$,
 where $l_j, m_j \in \mathbb{Z}$ and $\mu_j \in E_8(-2)$.
 Then,
 we have
 $$
 -2=(i(\nu_1)\cdot i(\nu_1))=2 l_1 m_1 +(\mu_1\cdot \mu_1).
 $$
 Since self-intersection numbers of $E_8(-2)$ are in $4\mathbb{Z}$,
 it follows that both $l_1$ and $m_1$ are  odd numbers.
 Similarly, both of $l_2$ and $m_2$ are  odd numbers.
 On the other hand,
 we obtain
 $$
 1=(i(\nu_1)\cdot i(\nu_2))= l_1 m_2 + l_2 m_1 +(\mu_1\cdot \mu_2).
 $$
 Here, $(\mu_1\cdot \mu_2)\in 2\mathbb{Z}$, because $\mu_j \in E_8(-2)$.
 Hence, $l_1m_2 + l_2 m_1 $ is an odd number.
 This is a contradiction.
 Therefore, there is no primitive embedding (\ref{iPrimEmb}).
 So, by
 Lemma \ref{LemNikulin},
 there is no Nikulin involution on the $K3$ surface (\ref{S(t)}). 
 \end{proof}

 \subsection{Sandwich phenomenon for the families of \cite{CMS} and \cite{MSY}}

 As a by-product of our explicit expression of Nikulin involutions,
 we can obtain a  handy description of a sandwich phenomenon for another families of  $K3$ surfaces.
 
 Clingher, Malmendier and Shaska \cite{CMS} obtains another natural extension of the story for the $K3$ surfaces ${\rm Kum}(\mathfrak{A})$ and $S_{{\rm CD}}$.
 They consider the family of lattice polarized $K3$ surfaces $S_{{\rm CMS}}$
with the  transcendental lattice 
$U\oplus U\oplus A_1(-1) \oplus A_1(-1).$
 Also, they study the relation between $S_{{\rm CMS}}$
 and $K_{{\rm MSY}}$,
 where $K_{{\rm MSY}}$ is the $K3$ surfaces with the transcendental lattice
 $U(2)\oplus U(2)\oplus A_1(-1) \oplus A_1(-1)$.
 The family  of $K_{{\rm MSY}}$ is  precisely studied by Matsumoto, Sasaki and Yoshida \cite{MSY} 
 on the basis of the viewpoint of the double covering of $\mathbb{P}^2(\mathbb{C})$ branched along six lines and hypergeometric differential equations. 
 In \cite{CMS},
 they show that there is a van Geemen-Sarti involution on $S_{{\rm CMS}}$ and the corresponding double covering $S_{{\rm CMS}} \dashrightarrow K_{{\rm MSY}}.$
They call the $K3$ surface $S_{{\rm CMS}}$  the  van Geemen-Sarti partner of $K_{{\rm MSY}}$.

In this section, we use the Weierstrass equation of the $K3$ surfaces $S_{{\rm CMS}}$  given in \cite{NU}:
 \begin{align}\label{NU}
 z^2= y^3 + (u_{5,3} x^5 + u_{4,4} x^4 + u_{3,5} x^3) y+ (u_{7,5} x^7 + u_{6,6} x^6 +u_{5,7} x^5),
 \end{align}
 Set
 \begin{align*}
 t_4 = u_{4,4}, \quad t_6=u_{6,6}, \quad t_8=u_{5,3} u_{3,5}, \quad t_{10}=u_{5,3}u_{5,7}+u_{3,5} u_{7,5}, \quad t_{12}=u_{5,7}u_{7,5}.
 \end{align*}
  Via the inverse of the period mapping for the family of (\ref{NU}),
  $t_4,t_6,t_8,t_{10}$ and $t_{12}$
  give modular forms on the $4$-dimensional symmetric domain $\mathcal{D}$ of type $IV$
  for the orthogonal group $O(2,4;\mathbb{Z})$.
  Also, there are $3$  modular forms for $O(2,4;\mathbb{Z})$ of the $3$ non-trivial characters (see \cite{NU} Theorem 1.1).
Now,  remark that
  the $K3$ surface $S_{{\rm CMS}}$ given by (\ref{NU}) degenerates to the $K3$ surface $S_{{\rm CD}}$ 
  if  $u_{5,3}=0$ and $u_{7,5}=1$.
  Thus, the family of $S_{{\rm CMS}}$,
  which is characterized by the connection with the configuration of six lines in $\mathbb{P}^2(\mathbb{C})$ and modular forms for $O(2,4;\mathbb{Z})$,
   gives another natural extension of the family of $S_{{\rm CD}}$.

 By the birational transformation
 $(x,y,z)=(\frac{Y}{u_{5,3} X+u_{7,5}},\frac{X Y^2}{(u_{5,3} X +u_{7,5})^2},\frac{Y^3 Z}{(u_{5,3} X+u_{7,5})^3}),$
the equation (\ref{NU}) is transformed to
\begin{align}\label{SCMS}
S_{{\rm CMS}}: Z^2=Y^3+u_{4,4} Y +u_{6,6} +\Big(  X + \frac{u_{5,3} u_{3,5} Y^2 + (u_{5,3} u_{5,7}+u_{3,5} u_{7,5}) Y +u_{5,7}u_{7,5}}{X}\Big). 
\end{align}
There is a Nikulin involution
 \begin{align}\label{iotaCMS}
 \iota_{\psi,{\rm CMS}}: (X,Y,Z)\mapsto \Big(\frac{u_{5,3} u_{3,5} Y^2 + (u_{5,3} u_{5,7}+u_{3,5} u_{7,5}) Y +u_{5,7}u_{7,5}}{X} ,Y,-Z\Big)
 \end{align}
 on the $K3$ surface (\ref{SCMS}).
 Note that this is equal to the Nikulin involution
 which appears in \cite{CMS}.
 As in the proof of Theorem \ref{ThmMaSand},
 the Nikulin involution $\iota_{\psi,{\rm CMS}}$
 induces a double covering
 $
\psi_{{\rm CMS}}: S_{{\rm CMS}} \dashrightarrow K_{{\rm MSY}},
 $
 where $K_{{\rm MSY}}$ is a $K3$ surface given by
 \begin{align}\label{KCMS}
 K_{{\rm MSY}}:Z^2=Y^3+u_{4,4} Y +u_{6,6} +\Big(  U^2 + \frac{u_{5,3} u_{3,5} Y^2 + (u_{5,3} u_{5,7}+u_{3,5} u_{7,5}) Y +u_{5,7}u_{7,5}}{U^2}\Big).  
 \end{align}
 The transcendental lattice of $K_{{\rm MSY}}$ is given by $U(2)\oplus U(2)\oplus A_1(-1) \oplus A_1(-1).$
 If $u_{5,3}=0$ and $u_{7,5}=1$, $K_{{\rm MSY}}$  degenerates to ${\rm Kum}(\mathfrak{A})$ of (\ref{Kum(A)}).
 So, we can regard the family of $K_{{\rm MSY}}$ as another extension of  $\mathcal{G}_{\rm Kum}$.
 There is a Nikulin involution
 \begin{align}\label{iotaMSY}
 \iota_{\varphi,{\rm MSY}}: (U,Y,Z)\mapsto  (-U,Y,Z)
 \end{align}
 which induces a double covering $\varphi_{{\rm MSY}}: K_{{\rm MSY}}\dashrightarrow S_{{\rm CMS}}$.
 Summarizing this argument, we have the following theorem.
 
 \begin{thm}\label{ThmMSYCMS}
 One has the sandwich for $K_{{\rm MSY}} $ of (\ref{SCMS}) and $S_{{\rm CMS}}$ of (\ref{KCMS}):
 \begin{align*}
K_{{\rm MSY}} \overset{\varphi_{{\rm MSY}}}\dashrightarrow S_{{\rm CMS}} \overset{\psi_{{\rm CMS}}}\dashrightarrow K_{{\rm MSY}}.
\end{align*}
Here, $\psi_{\rm CMS}$  ($\varphi_{\rm MSY}$ , resp.) is the double covering given by the involution (\ref{iotaCMS}) ((\ref{iotaMSY}), resp.).
 \end{thm}

 \section*{Acknowledgement}
 The first author  is supported by JSPS Grant-in-Aid for Scientific Research (18K13383)
and 
MEXT LEADER.
The second author is supported by JSPS Grant-in-Aid for Scientific Research (19K03396).
The authors are thankful to Professor Jiro Sekiguchi for valuable discussions about the surface $\Sigma(t)$ in Section 4 from the view point of complex reflection groups,
and the anonymous referee for suggestions for improvement.

\vspace{3mm}

\begin{center}
\hspace{7.7cm}\textit{Atsuhira  Nagano}\\
\hspace{7.7cm}\textit{Faculty of Mathematics and Physics}\\
 \hspace{7.7cm} \textit{Institute of Science and Engineering}\\
\hspace{7.7cm}\textit{Kanazawa University}\\
\hspace{7.7cm}\textit{Kakuma, Kanazawa, Ishikawa}\\
\hspace{7.7cm}\textit{920-1192, Japan}\\
 \hspace{7.7cm}\textit{(E-mail: atsuhira.nagano@gmail.com)}
  \end{center}

\vspace{1mm}

\begin{center}
\hspace{7.7cm}\textit{Hironori  Shiga}\\
\hspace{7.7cm}\textit{Graduate School of Science }\\
\hspace{7.7cm}\textit{Chiba University}\\
\hspace{7.7cm}\textit{Yayoi-cho 1-33, Inage-ku, Chiba}\\
\hspace{7.7cm}\textit{263-8522, Japan}\\
 \hspace{7.7cm}\textit{(E-mail: pigrohironoris@gmail.com)}
  \end{center}

\end{document}